\documentclass[12pt, a4paper]{amsart}
\usepackage{amscd, amssymb, pb-diagram}
\usepackage{amsthm,tikz}
\usepackage{empheq}

\numberwithin{equation}{section}
\def\clsp{\overline{\operatorname{span}}}

\def\supp{\operatorname{supp}}

\def\iso{\operatorname{Iso}}

\def\max{\operatorname{max}}

\def\supp{\operatorname{supp}}

\newcommand{\grp}{\operatorname{\mathcal{G}}}
\newcommand{\reg}{{\operatorname{reg}}}
\newcommand{\sing}{{\operatorname{sing}}}
\newcommand{\dom}{{\operatorname{dom}}}
\newcommand{\ran}{{\operatorname{ran}}}
\newcommand{\is}{{\operatorname{iso}}}

\newcommand{\at}[1]{\overline{\mathrm{t}#1}}
\newcommand{\prim}{\operatorname{Prim}^\tau}

\def\C{\mathbb{C}}

\def\R{\mathbb{R}}
\def\N{\mathbb{N}}
\def\Z{\mathbb{Z}}
\def\T{\mathbb{T}}

\def\GG{\mathcal{G}}

\def\KK{\mathcal{K}}

\def\GG{\mathcal{G}}

\def\DD{\mathcal{D}}

\def\DD{\mathcal{D}}

\newtheorem{thm}{Theorem}[section]
\newtheorem{cor}[thm]{Corollary}

\newtheorem{lemma}[thm]{Lemma}
\newtheorem{prop}[thm]{Proposition}

\theoremstyle{definition}
\newtheorem{definition}[thm]{Definition}

\newtheorem{notation}[thm]{Notation}

\theoremstyle{remark}
\newtheorem{remark}[thm]{Remark}

\newtheorem{example}[thm]{Example}

\begin{document}

\title[Graph algebras and orbit equivalence]{Graph algebras and orbit equivalence}

\author{Nathan Brownlowe}
\address{Nathan Brownlowe and Michael F. Whittaker \\ School of Mathematics and
Applied Statistics  \\
The University of Wollongong\\
NSW  2522\\
AUSTRALIA} 
\email{nathanb@uow.edu.au, mfwhittaker@gmail.com}
\author{Toke Meier Carlsen}
\address{Toke Meier Carlsen \\ Department of Mathematical Sciences \\ Norwegian University of Science and Technology (NTNU) \\
7491 Trondheim \\
NORWAY}
\email{Toke.Meier.Carlsen@math.ntnu.no}
\author{Michael F. Whittaker}

\thanks{This research was supported by the Australian Research Council.}


\begin{abstract}
We introduce the notion of orbit equivalence of directed graphs,  
following Matsumoto's notion of 
continuous orbit equivalence for topological Markov shifts. We show that two 
graphs in which every cycle has an exit are orbit equivalent if and only if 
there is a diagonal-preserving isomorphism between their $C^*$-algebras. We 
show that it is necessary to assume that every cycle has an exit for the 
forward implication, but that the reverse implication holds for arbitrary 
graphs. As part of our 
analysis of arbitrary graphs $E$ we construct a groupoid 
$\grp_{(C^*(E),\mathcal{D}(E))}$ from the graph algebra $C^*(E)$ and its 
diagonal subalgebra $\DD(E)$ which generalises Renault's Weyl groupoid 
construction applied to $(C^*(E),\DD(E))$. We show that 
$\grp_{(C^*(E),\mathcal{D}(E))}$ recovers the graph 
groupoid $\GG_E$ without the assumption that every cycle in $E$ has an exit, 
which is required to apply Renault's results to $(C^*(E),\DD(E))$. We finish 
with applications of our results to out-splittings of graphs and to amplified 
graphs.
\end{abstract}

\maketitle

\section{Introduction}\label{sec: intro}

The relationship between orbit equivalence and isomorphism of $C^*$-algebras 
has been studied extensively in the last 20 years. 
The first result of this type was the celebrated theorem of Giordano, Putnam 
and Skau \cite[Theorem 2.4]{GPS}, in which they showed that orbit equivalence for minimal dynamical systems on the Cantor set is equivalent to isomorphism of their corresponding crossed product $C^*$-algebras. 
The importance of Giordano, Putnam and Skau's result cannot be overstated. 
In general there is no direct method of checking whether two Cantor 
minimal systems are orbit equivalent. However, because the crossed product 
$C^*$-algebras are classifiable, Giordano, Putnam and Skau's result means that 
orbit equivalence can be determined using 
$K$-theory. The work in \cite{GPS} has been generalised in 
many directions, including Tomiyama's results on topologically free dynamical 
systems on compact Hausdorff spaces 
\cite{Tom}, and Giordano, Matui, Putnam and Skau's extension of \cite[Theorem 2.4]{GPS} to minimal 
$\Z^d$-actions on the Cantor set \cite{GMPS}. 

More recently, Matsumoto and Matui have shown in \cite{MM} that two 
irreducible one-sided topological Markov shifts $(X_A,\sigma_A)$ and 
$(X_B,\sigma_B)$ are continuously orbit equivalent if and only if the 
corresponding Cuntz-Krieger algebras $\mathcal{O}_A$ and $\mathcal{O}_B$ are 
isomorphic and $\det(I-A)=\det(I-B)$. The proof of Matsumoto and Matui's 
theorem relies on two key results. The first of these is \cite[Theorem 
1.1]{Mat}, in which Matsumoto proves that the following statements are 
equivalent:
\begin{enumerate}
\item $(X_A,\sigma_A)$ and $(X_B,\sigma_B)$ are continuously orbit equivalent,
\item there exists a $*$-isomorphism $\phi:\mathcal{O}_A\to\mathcal{O}_B$ which maps the maximal abelian subalgebra $\mathcal{D}_A$ onto $\mathcal{D}_B$, and 
\item the topological full group of $(X_A,\sigma_A)$ and the topological full group of $(X_B,\sigma_B)$ are spatially isomorphic. (In \cite[Theorem 1.1]{Mat3}, Matsumoto showed that this is equivalent to the topological full groups being abstractly isomorphic.)
\end{enumerate}
The second key result is \cite[Proposition 4.13]{Ren2}, which, as noticed by 
Matui (see \cite[Theorem 5.1]{Matui}), implies that there exists a 
$*$-isomorphism $\phi:\mathcal{O}_A\to\mathcal{O}_B$ that maps the maximal 
abelian subalgebra, or diagonal, $\mathcal{D}_A$ onto $\mathcal{D}_B$ if and 
only if the corresponding groupoids $\mathcal{G}_A$ and $\mathcal{G}_B$ are 
isomorphic.

In this paper we initiate the study of orbit equivalence of directed graphs, 
and we prove the analogous result to \cite[Proposition 4.13]{Ren2} for graph 
algebras. In particular, as part of our main result we prove that if $E$ and 
$F$ are two graphs in which every cycle has an exit, then the following are 
equivalent:
\begin{enumerate}
\item[(1)] There is an isomorphism from $C^*(E)$ to $C^*(F)$ which maps the 
diagonal subalgebra $\mathcal{D}(E)$ onto $\mathcal{D}(F)$.
\item[(2)] The graph groupoids $\grp_E$ and $\grp_E$ are isomorphic as 
topological groupoids.
\item[(3)] The pseudogroups of $E$ and $F$ are isomorphic.
\item[(4)] The graphs $E$ and $F$ are orbit equivalent.
\end{enumerate}

It is natural to ask whether every cycle having an exit is necessary for our 
results. In our main result we in fact prove that $(1)\iff (2)$, $(3)\iff (4)$ 
and $(2)\implies (3)$ all hold for arbitrary directed graphs. It is only the 
implication $(3)\implies (2)$ that requires that every cycle has an exit (and we provide examples that show that $(3)\implies (2)$ does not hold in general without the assumption that every cycle has an exit). Our 
analysis of these implications for arbitrary graphs provides our most technical 
innovation, which 
is the introduction of a groupoid $\grp_{(C^*(E),\DD(E))}$ associated to 
$(C^*(E),\DD(E))$ that we call the extended Weyl groupoid. Our construction 
generalises Renault's Weyl groupoid construction from 
\cite[Definition~4.11]{Ren2} applied to $(C^*(E),\DD(E))$. We show that 
$\grp_{(C^*(E),\DD(E))}$ and $\grp_E$ are isomorphic as topological groupoids 
for an arbitrary graph $E$, which can be deduced from 
Renault's results in \cite{Ren2} only when every cycle in $E$ has an exit.

We conclude our paper with two applications of our main theorem. Our 
first application shows that if two general graphs $E$ and $F$ are conjugate 
then there is an isomorphism from $C^*(E)$ to $C^*(F)$ which maps 
$\mathcal{D}(E)$ onto $\mathcal{D}(F)$. As a corollary, we strengthen a result 
of Bates and Pask \cite[Theorem 3.2]{BP} on out-splitting of graphs. Our 
second application adds three additional equivalences to Eilers, Ruiz, and S\o 
rensen's complete invariant for amplified graphs \cite[Theorem 1.1]{ERS}.

The paper is organized as follows. Section~\ref{sec: Directed graphs and orbit 
equivalence} provides background on graphs, their groupoids and their 
$C^*$-algebras. In Section~\ref{sec: oe and pseudo} we define orbit 
equivalence of graphs and associate with each graph a pseudogroup which is the analogue of the topological full group Matsumoto has associated with each irreducible one-sided topological Markov shift, and we show that two graphs are orbit 
equivalent if and only if their pseudogroups are isomorphic. In 
Section~\ref{sec: ext Weyl} we construct the extended Weyl groupoid 
$\grp_{(C^*(E),\DD(E))}$ from $(C^*(E),\DD(E))$, and we show that 
$\grp_{(C^*(E),\DD(E))}$ and $\grp_E$ are isomorphic as topological groupoids. 
We use this result to show that if there is 
a diagonal-preserving isomorphism from $C^*(E)$ to $C^*(F)$, then $\grp_E$ and 
$\grp_F$ are isomorphic as topological 
groupoids. In Section~\ref{sec: Statement of the main result} we finish the proof of our main 
theorem and 
provide examples. Finally, in Section~\ref{sec: applications} we give the two 
applications of our main theorem.

\begin{remark}
We have learned that Xin Li has also considered orbit equivalence for directed graphs, and has independently proved that two graphs in which every cycle has an exit are orbit equivalent if and only if there is a diagonal-preserving isomorphism between their C*-algebras.
\end{remark}

\section{Background on the groupoids and $C^*$-algebras of directed 
graphs}\label{sec: Directed graphs and orbit equivalence}

We begin with some background on graphs and their $C^*$-algebras. In this section we recall the definitions of the boundary path space of a directed graph, graph $C^*$-algebras and graph groupoids. 

\subsection{Graphs and their $C^*$-algebras}
We refer the reader to \cite{Raeburn2005} for a more detailed treatment on graphs and their $C^*$-algebras. However, we note that the directions of arrows defining a graph are reversed in this paper. We used this convention so that our results can easily be compared with the work of Matsumoto and Matui's work on shift spaces.

A {\em directed graph} (also called a \emph{quiver}) $E=(E^0,E^1,r,s)$ consists of countable sets $E^0$ and $E^1$, and range and source maps $r,s:E^1\to E^0$. The elements of $E^0$ are called vertices, and the elements of $E^1$ are called edges. 

A {\em path} $\mu$ of length $n$ in $E$ is a sequence of edges $\mu=\mu_1\dots\mu_n$ such that $r(\mu_i)=s(\mu_{i+1})$ for all $1\le i\le n-1$. The set of paths of length $n$ is denoted $E^n$. We denote by $|\mu|$ the length of $\mu$. The range and source maps extend naturally to paths: $s(\mu):=s(\mu_1)$ and $r(\mu):=r(\mu_n)$. We regard the elements of $E^0$ as path of length 0, and for $v \in E^0$ we set $s(v):=r(v):=v$. For $v\in E^0$ and $n\in\N$ we denote by $vE^n$ the set of paths of length $n$ with source $v$, and by $E^nv$ the paths of length $n$ with range $v$. We define $E^*:=\bigcup_{n\in\N}E^n$ to be the collection of all paths with finite length. For $v,w\in E^0$ let $vE^*w:=\{\mu\in E^*:s(\mu)=v\text{ and }r(\mu)=w\}$. We define $E^0_\reg:=\{v\in E^0:vE^1\text{ is finite and nonempty}\}$ and $E^0_\sing:=E^0\setminus E^0_\reg$. If $\mu=\mu_1\mu_2\cdots\mu_m, \nu=\nu_1\nu_2\cdots\nu_n\in E^*$ and $r(\mu)=s(\nu)$, then we let $\mu\nu$ denote the path $\mu_1\mu_2\cdots\mu_m\nu_1\nu_2\cdots\nu_n$. 

A {\em loop} (also called a \emph{cycle}) in $E$ is a path $\mu\in E^*$ such that $|\mu|\ge 1$ and $s(\mu)=r(\mu)$. If $\mu$ is a loop and $k$ is a positive integer, then $\mu^k$ denotes the loop $\mu\mu\cdots\mu$ where $\mu$ is repeated $k$-times. We say that the loop $\mu$ is \emph{simple} if $\mu$ is not equal to $\nu^k$ for any loop $\nu$ and any integer $k\ge 2$. Notice than any loop $\mu$ is equal to $\nu^k$ for some simple loop $\nu$ and some positive integer $k$. An edge $e$ is an {\em exit} to the loop $\mu$ if there exists $i$ such that $s(e)=s(\mu_i)$ and $e\not=\mu_i$. A graph is said to satisfy \emph{condition (L)} if every loop has an exit. 

A {\em Cuntz-Krieger $E$-family} $\{P,S\}$ consists of a set of mutually orthogonal projections $\{P_v:v\in E^0\}$ and partial isometries $\{S_e:e\in E^1\}$ satisfying
\begin{enumerate}
\item[(CK1)] $S_e^*S_e=P_{r(e)}$ for all $e\in E^1$;
\item[(CK2)] $S_eS_e^*\le P_{s(e)}$ for all $e\in E^1$;
\item[(CK3)] $\displaystyle{P_v=\sum_{e\in vE^1}S_eS_e^*}$ for all $v\in E^0_\reg$.
\end{enumerate}
The {\em graph $C^*$-algebra} $C^*(E)$ is the universal $C^*$-algebra generated by a Cuntz-Krieger $E$-family. We denote by $\{p,s\}$ the Cuntz-Krieger $E$-family generating $C^*(E)$. There is a strongly continuous action $\gamma:C^*(E)\to\T$, called the {\em gauge action}, satisfying $\gamma_z(p_v)=p_v$ and $\gamma_z(s_e)=zs_e$, for all $z\in \T$, $v\in E^0$, $e\in E^1$. If $\{Q,T\}$ is a Cuntz-Krieger $E$-family in a $C^*$-algebra $B$, then we denote by $\pi_{Q,T}$ the homomorphism $C^*(E)\to B$ such that $\pi_{Q,T}(p_v)=Q_v$ for all $v\in E^0$, and $\pi_{Q,T}(s_e)=T_e$ for all $e\in E^1$. an Huef and Raeburn's \emph{gauge invariant uniqueness theorem} \cite{aHR} says that $\pi_{Q,T}$ is injective if and only if there is an action $\beta$ of $\T$ on the $C^*$-algebra generated by $\{Q,T\}$ satisfying $\beta_z(Q_v)=Q_v$ and $\beta_z(T_e)=zT_e$, for all $z\in \T$, $v\in E^0$, $e\in E^1$, and $Q_v\ne 0$ for all $v\in E^0$.

If $\mu=\mu_1\cdots\mu_n\in E^n$ and $n\ge 2$, then we let $s_\mu:=s_{\mu_1}\cdots s_{\mu_n}$. Likewise, we let $s_v:=p_v$ if $v\in E^0$. Then $C^*(E)=\clsp\{s_\mu s_\nu^*:\mu,\nu\in E^*,\ r(\mu)=r(\nu)\}$. The $C^*$-subalgebra $\DD(E):=\clsp\{s_\mu s_\mu^*:\mu\in E^*\}$ of $C^*(E)$ is a maximal abelian subalgebra if and only if every loop in $E$ has an exit (see \cite[Example 3.3]{NR}).

\subsection{The boundary path space of a graph}

An \emph{infinite path} in $E$ is an infinite sequence $x_1x_2\dots $ of edges in $E$ such that $r(e_i)=s(e_{i+1})$ for all $i$. We let $E^\infty$ be the set of all infinite paths in $E$. The source map extends to $E^\infty$ in the obvious way. We let $|x|=\infty$ for $x\in E^\infty$. The \emph{boundary path space} of $E$ is the space 
\begin{equation*}
	\partial E:=E^\infty\cup\{\mu\in E^*:r(\mu)\in E^0_\sing\}.
\end{equation*} 
If $\mu=\mu_1\mu_2\cdots\mu_m\in E^*$, $x=x_1x_2\cdots\in E^\infty$ and $r(\mu)=s(x)$, then we let $\mu x$ denote the infinite path  $\mu_1\mu_2\cdots\mu_m x_1x_2\cdots \in E^\infty$.

For $\mu\in E^*$, the \emph{cylinder set} of $\mu$ is the set
\begin{equation*}
	Z(\mu):=\{\mu x\in\partial E:x\in r(\mu)\partial E\},
\end{equation*} 
where $r(\mu)\partial E:=\{x \in \partial E : r(\mu)=s(x)\}$. Given $\mu\in E^*$ and a finite subset $F\subseteq r(\mu)E^1$ we define
\begin{equation*}
	Z(\mu\setminus F):=Z(\mu)\setminus\left(\bigcup_{e\in F}Z(\mu e)\right).
\end{equation*}

The boundary path space $\partial E$ is a locally compact Hausdorff space with the topology given by the basis $\{Z(\mu\setminus F): \mu\in E^*,\ F\text{ is a finite subset of }r(\mu)E^1\}$, and each such $Z(\mu\setminus F)$ is compact and open (see \cite[Theorem 2.1 and Theorem 2.2]{Web}). Moreover, \cite[Theorem 3.7]{Web} shows that there is a unique homeomorphism $h_E$ from $\partial E$ to the spectrum of $\DD(E)$ given by
\begin{equation}\label{he_map}
	h_E(x)(s_\mu s_\mu^*)=\begin{cases}
		1&\text{if }x\in Z(\mu),\\
		0&\text{if }x\notin Z(\mu).
	\end{cases}
\end{equation}

Our next lemma gives a description of the topology on the boundary path space, which we will need in the proof of Proposition \ref{prop:pseudo}.

\begin{lemma}\label{lem:open}
	Every nonempty open subset of $\partial E$ is the disjoint union of sets that are both compact and open.
\end{lemma}

\begin{proof}
Let $U$ be a nonempty open subset of $\partial E$. For each $x\in U$ let 
\[
B_x:=\{(\mu,F): \mu\in E^*,\ F\text{ is a finite subset of }r(\mu)E^1,\ x\in Z(\mu\setminus F)\subseteq U\}.
\]
If $(\mu,F)\in B_x$, then $x\in Z(\mu)$ and $x\notin Z(\mu e)$ for each $e\in F$. Let $\mu_x$ be the shortest $\mu\in E^*$ such that $(\mu,F)\in B_x$ for some finite subset $F$ of $r(\mu)E^1$, and let $F_x:=\cap\{F: (\mu_x,F)\in B_x\}$. Then $(\mu_x,F_x)\in B_x$ and $Z(\mu\setminus F)\subseteq Z(\mu_x\setminus F_x)$ for all $(\mu,F)\in B_x$. It follows that if $x,y\in U$, then either $Z(\mu_x\setminus F_x)=Z(\mu_y\setminus F_y)$ or $Z(\mu_x\setminus F_x)\cap Z(\mu_y\setminus F_y)=\emptyset$. Since $U=\cup_{x\in U}Z(\mu_x\setminus F_x)$ and each $Z(\mu_x\setminus F_x)$ is open and compact, this shows that $U$ is the disjoint union of sets that are both compact and open.
\end{proof}

For $n\in\N$, let $\partial E^{\ge n}:=\{x\in\partial E: |x|\ge n\}$. Then $\partial E^{\ge n}= \cup_{\mu \in E^n} Z(\mu)$ is an open subset of $\partial E$. We define the \emph{shift map} on $E$ to be the map $\sigma_E:\partial E^{\ge 1}\to\partial E$ given by $\sigma_E(x_1x_2x_3\cdots)=x_2x_3\cdots$ for $x_1x_2x_3\cdots\in\partial E^{\ge 2}$ and $\sigma_E(e)=r(e)$ for $e\in\partial E\cap E^1$. For $n\ge 1$, we let $\sigma_E^n$ be the $n$-fold composition of $\sigma_E$ with itself. We let $\sigma_E^0$ denote the identity map on $\partial E$. Then $\sigma_E^n$ is a local homeomorphism for all $n\in\N$. When we write $\sigma_E^n(x)$, we implicitly assume that $x\in\partial E^{\ge n}$.  

We say that $x\in\partial E$ is \emph{eventually periodic} if there are $m,n\in\N$, $m\ne n$ such that $\sigma_E^m(x)=\sigma_E^n(x)$. Notice that $x\in\partial E$ is eventually periodic if and only if $x=\mu\nu\nu\nu\cdots$ for some path $\mu\in E^*$ and some loop $\nu\in E^*$ with $s(\nu)=r(\mu)$. By replacing $\nu$ by a subloop if necessary, we can assume that $\nu$ is a simple loop.

\subsection{Graph groupoids}

In \cite{KPRR}, Kumjian, Pask, Raeburn, and Renault defined groupoid $C^*$-algebras associated to a locally-finite directed graph with no sources. Their construction has been generalized to compactly aligned topological $k$-graphs in \cite{Yee}. We will now explain this construction in the case that $E$ is an arbitrary graph. The resulting groupoid is isomorphic to the one constructed by Paterson in \cite{Pat}. Let  
\begin{equation*}
	\grp_E:=\{(x,m-n,y): x,y\in\partial E,\ m,n\in\N,\text{ and } \sigma^m(x)=\sigma^n(y)\},
\end{equation*}
with product $(x,k,y)(w,l,z):=(x,k+l,z)$ if $y=w$ and undefined otherwise, and inverse given by $(x,k,y)^{-1}:=(y,-k,x)$.
With these operations $\grp_E$ is a groupoid (cf. \cite[Lemma 2.4]{KPRR}). The unit space $\grp_E^0$ of $\grp_E$ is $\{(x,0,x):x\in\partial E\}$ which we will freely identify with $\partial E$ via the map $(x,0,x)\mapsto x$ throughout the paper. We then have that the range and source maps $r,s: \grp_E \to \partial E$ are given by $r(x,k,y)=x$ and $s(x,k,y)=y$.

We now define a topology on $\grp_E$. Suppose $m,n\in\N$ and $U$ is an open subset of $\partial E^{\ge m}$ such that the restriction of $\sigma_E^m$ to $U$ is injective, $V$ is an open subset of $\partial E^{\ge n}$ such that the restriction of $\sigma_E^n$ to $V$ is injective, and that $\sigma_E^m(U)=\sigma_E^n(V)$, then we define
\begin{equation}\label{grp_topology}
	Z(U,m,n,V):=\{(x,k,y)\in\grp_E:x\in U,\ k=m-n,\ y\in V,\ \sigma_E^m(x)=\sigma_E^n(y)\}.
\end{equation}
Then $\grp_E$ is a locally compact, Hausdorff, \'{e}tale topological groupoid with the topology generated by the basis consisting of sets $Z(U,m,n,V)$ described in \eqref{grp_topology}, see \cite[Proposition 2.6]{KPRR} for an analogous situation. For $\mu,\nu \in E^*$ with $r(\mu)=r(\nu)$, let $Z(\mu,\nu):=Z(Z(\mu),|\mu|,|\nu|,Z(\nu))$. It follows that each $Z(\mu,\nu)$ is compact and open, and that the topology $\partial E$ inherits when we consider it as a subset of $\grp_E$ by identifying it with $\{(x,0,x):x\in\partial E\}$ agrees with the topology described in the previous section. Notice that for all $\mu,\nu \in E^*$, $U$ a compact open subset of $Z(\mu)$, and $V$ a compact open subset of $Z(\nu)$, the collection $\{Z(U,|\mu|,|\nu|,V): \sigma_E^{|\mu|}(U)=\sigma_E^{|\nu|}(V)\}$ is a basis for the topology of $\grp_E$. According to \cite[Proposition 6.2]{Yee}, $\grp_E$ is topologically amenable in the sense of \cite[Definition 2.2.8]{AR}. It follows from \cite[Proposition 3.3.5]{AR} and \cite[Proposition 6.1.8]{AR} that the reduced and universal $C^*$-algebras of $\grp_E$ are equal, and we denote this $C^*$-algebra by $C^*(\grp_E)$.

\begin{prop}[{Cf. \cite[Proposition 4.1]{KPRR}}]\label{prop:groupoid}
Suppose $E$ is a graph. Then there is a unique isomorphism $\pi:C^*(E) \to C^* 
(\grp_E)$ such that $\pi(p_v)=1_{Z(v,v)}$ for all $v \in E^0$ and 
$\pi(s_e)=1_{Z(e,r(e))}$ for all $e \in E^1$, and such that 
$\pi(\mathcal{D}(E))=C_0(\grp_E^0)$.
\end{prop}

\begin{proof}
Using calculations along the lines of those used in the proof of \cite[Proposition 4.1]{KPRR}, it is straight forward to check that \[
\{Q,T\}:=\{Q_v:=1_{Z(v,v)} \text{ and } T_e:=1_{Z(e,r(e))} : v \in E^0, e \in E^1\}
\]
is a Cuntz-Krieger $E$-family. The universal property of $\{p,s\}$ implies 
that there is a $*$-homomorphism $\pi:=\pi_{Q,T}:C^*(E)\to C^* (\grp_E)$ 
satisfying $\pi(p_v)=Q_v$ and $\pi(s_e)=T_e$. An argument similar to the one 
used in the proof of \cite[Proposition 4.1]{KPRR} shows that $C^* (\grp_E)$ is 
generated by $\{Q,T\}$, so $\pi$ is surjective. The cocycle $(x,k,y)\mapsto k$ 
induces an action $\beta$ of $\T$ on $C^* (\grp_E)$ satisfying 
$\beta_z(Q_v)=Q_v$ and $\beta_z(T_e)=zT_e$, for all $z\in \T$, $v\in E^0$, 
$e\in E^1$ (see \cite[Proposition II.5.1]{Ren}), and since $Q_v=1_{Z(v,v)}\ne 
0$ for all $v\in E^0$, the gauge invariant uniqueness theorem of $C^*(\grp_E)$ 
(\cite[Theorem 2.1]{BHRS}) implies that $\pi$ is injective. Since 
$\mathcal{D}(E)$ is generated by $\{s_\mu s_\mu^*:\mu\in E^*\}$ and $\pi(s_\mu 
s_\mu^*)=1_{Z(\mu,\mu)}$, we have that $\pi$ maps $\mathcal{D}(E)$ into 
$C_0(\grp_E^0)$. An application of the Stone-Weierstrass theorem implies that 
$C_0(\grp_E^0)$ is generated by $\{1_{Z(\mu,\mu)}:\mu\in E^*\}$. Hence 
$\pi(\mathcal{D}(E))=C_0(\grp_E^0)$.
\end{proof}

Suppose $\grp$ is a groupoid, the isotropy group of $x\in \grp^0$ is the group $\iso(x):=\{\gamma\in\grp:s(\gamma)=r(\gamma)=x\}$. In \cite{Ren2}, an \'{e}tale groupoid is said to be \emph{topologically principal} if the set of points of $\grp^0$ with trivial isotropy group is dense. We will now characterize when $\grp_E$ is topologically principal.

\begin{prop} \label{prop:topological principal}
	Let $E$ be a graph. Then the graph groupoid $\grp_E$ is topologically principal if and only if every loop in $E$ has an exit.
\end{prop}

\begin{proof}
	Let $x\in \partial E$. We claim that $(x,0,x)$ has nontrivial isotropy group if and only if $x$ is eventually periodic. Indeed, suppose $(x,m-n,x) \in \iso(x)$ with $m\neq n$, then $\sigma^m(x)=\sigma^n(x)$ and $x$ is eventually periodic. On the other hand, suppose $x=\mu \lambda^\infty$, then $(x,(|\mu|+|\lambda|)-|\mu|,x) \in \iso(x)$, proving the claim.
Now observe that if $v$ is a vertex such that there are two different simple loops $\alpha$ and $\beta$ with $s(\alpha)=s(\beta)=v$, then any cylinder set $Z(\delta)$ for which $r(\delta)E^*v\ne\emptyset$ contains a $y$ such that $(y,0,y)$ has trivial isotropy. To see this, pick $\lambda \in r(\delta)E^*v$, then $y=\delta\lambda\alpha\beta\alpha^2\beta\alpha^3 \beta \cdots$ has trivial isotropy since it is not eventually periodic.
	
Assume that every loop in $E$ has an exit and suppose for contradiction that $U$ is an nonempty open subset of $\partial E$ such that $(x,0,x)$ has nontrivial isotropy group for every $x\in U$. Note that $U \subseteq E^\infty$ since $y \in \partial E$ with $|y| < \infty$ implies that the isotropy group of $(y,0,y)$ is trivial. Let $x\in U$. Since $x$ has nontrivial isotropy group, there exist $\zeta_1\in E^*$ and a loop $\eta$ such that $\zeta_1\eta^\infty \in Z(\zeta_1\eta^k) \subseteq U$ for some $k \in \N$. Since $\eta$ has an exit and $(x,0,x)$ has nontrivial isotropy group for every $x\in U$, it follows that there is a $\zeta_2\in r(\zeta_1)E^*$ such that $Z(\zeta_1\zeta_2)\subseteq U$ and such that $r(\zeta_2)E^*r(\zeta_1)=\emptyset$, for otherwise there would be two distinct simple loops based at $r(\zeta_1)$. By repeating this argument we get a sequences of paths $\zeta_1,\zeta_2,\zeta_3,\dots$ such that $s(\zeta_{n+1})=r(\zeta_n)$, $r(\zeta_{n+1})E^*r(\zeta_n)=\emptyset$ and $Z(\zeta_1\zeta_2\dots\zeta_n)\subseteq U$ for all $n$. The element $y=\zeta_1\zeta_2\zeta_3\dots$ then belongs to $U$, but since it only visits each vertex a finite number of times, $(y,0,y)$ must have trivial isotropy, which contradicts the assumption that $(x,0,x)$ has nontrivial isotropy group for every $x\in U$. Thus, $\grp_E$ is topologically principal if every loop in $E$ has an exit.
	
	Conversely, if $\mu$ is a loop without exit and $x=\mu\mu\mu\dots$, then $(x,0,x)$ is an isolated point in $\grp_E^0$ with nontrivial isotropy group. Thus, $\grp_E$ is not topologically principal if there is a loop in $E$ without an exit.
\end{proof}

Since the reduced and universal $C^*$-algebras of $\grp_E$ are equal, it follows from \cite[Proposition II.4.2(i)]{Ren} that we can regard $C^*(\grp_E)$ as a subset of $C_0(\grp_E)$. For $f\in C^*(\grp_E)$ and $j\in\Z$, we let $\Phi_j(f)$ denote the restriction of $f$ to $\{(x,k,y)\in\grp_E:k=j\}$, and for $m\in\N$ we let $\Sigma_m(f):=\sum_{j=-m}^m(1-\frac{|j|}{m+1})\Phi_j(f)$.

\begin{prop} \label{prop:cesaro}
	Let $E$ be a graph and let $f\in C^*(\grp_E)$. Then each $\Phi_k(f)$ and each $\Sigma_m(f)$ belong to $C^*(\grp_E)$, and $(\Sigma_m(f))_{m\in\N}$ converges to $f$ in $C^*(\grp_E)$.
\end{prop}

\begin{proof}
	Let $j\in\Z$. The map $(x,k,y)\mapsto k$ is a continuous cocycle from $\grp_E$ to $\Z$. For each $z\in\T$ there is a unique automorphism $\gamma_z$ on $C^*(\grp_E)$ such that $\gamma_z(g)(x,k,y)=z^kg(x,k,y)$ for $g\in C^*(\grp_E)$ and $(x,k,y)\in\grp_E$, and that the map $z\mapsto\gamma_z$ is a strongly continuous action of $\T$ on $C^*(\grp_E)$ (see \cite[Proposition II.5.1]{Ren}). It follows that the integral $\int_\T\gamma_z(f)z^{-j}\ dz$, where $dz$ denotes the normalized Haar measure on $\T$, is well-defined and belongs to $C^*(\grp_E)$ (see for example \cite[Section C.2]{RW}). Let $(x,k,y)\in\grp_E$. If $k\ne j$, then 
	\begin{equation*}
		\int_\T\gamma_z(f)z^{-j}\ dz(x,k,y)=\int_\T z^{k-j}\ dz\ f(x,k,y)=0,
	\end{equation*}
	and if $k=j$, then 
	\begin{equation*}
		\int_\T\gamma_z(f)z^{-j}\ dz(x,k,y)=\int_\T z^{k-j}\ dz\ f(x,k,y)=f(x,k,y).
	\end{equation*}
	Thus, $\Phi_j(f)=\int_\T\gamma_z(f)z^{-j}\ dz$ from which it follows that $\Phi_j(f)\in C^*(\grp_E)$.
	
	Since each $\Sigma_m(f)$ is a linear combination of functions of the form $\Phi_j(f)$, each $\Sigma_m(f)$ belongs to $C^*(\grp_E)$.
	
	For $m\in\N$, let $\sigma_m:\T\to\R$ be the Fej\'er's kernel defined by 
\[ 
\sigma_m(z)=\sum_{j=-m}^m(1-\frac{|j|}{m+1})z^{-j}.
\]
Then $\sigma_m(z)\ge 0$ for all $z\in\T$, $\int_\T\sigma_m(z)\ dz=1$, and 
	\begin{align*}
		\Sigma_m(f)&=\sum_{j=-m}^m\left(1-\frac{|j|}{m+1}\right)\Phi_j(f)\\
		&=\sum_{j=-m}^m\left(1-\frac{|j|}{m+1}\right)\int_\T\gamma_z(f)z^{-j}\ dz
		=\int_\T\gamma_z(f)\sigma_m(z)\ dz.
	\end{align*}
	Thus
	\begin{equation*}
		\|\Sigma_m(f)\|\le\int_\T\|\gamma_z(f)\|\sigma_m(z) dz =\|f\|.
	\end{equation*}
	If $g\in C_c(\grp_E)$, then there is an $m_0\in\N$ such that $\Phi_j(g)=0$ for $|j|>m_0$. It follows that 
	\begin{align*}
		\|g-\Sigma_m(g)\|&=\Big\|g-\sum_{j=-m}^m\left(1-\frac{|j|}{m+1}\right)\Phi_j(g)\Big\| \\
		&\leq \Big\|g-\sum_{j=-m}^m \Phi_j(g)\Big\|+\Big\|\sum_{j=-m}^m\left(\frac{|j|}{m+1}\right)\Phi_j(g)\Big\| \\
		&\leq \Big\|\sum_{j=-m}^m\left(\frac{|j|}{m+1}\right)\Phi_j(g)\Big\| \\
		&\leq \Big\|\sum_{j=-m_0}^{m_0}\left(\frac{|j|}{m+1}\right)\Phi_j(g)\Big\| \quad \text{ for } m \geq m_0 \\
		&\le\sum_{j=-m_0}^{m_0}\frac{|j|}{m+1}\|\Phi_j(g)\|\to 0 \quad \text{ as $m\to\infty$.}
	\end{align*}
Thus, for any $\epsilon>0$ there exists $g\in C_c(\grp_E)$ and an $M\in\N$ such that $||f-g||<\epsilon/3$ and $||g-\Sigma_m(g)||<\epsilon/3$ for any $m\ge M$, and then 
	\begin{equation*}
		||f-\Sigma_m(f)||\le ||f-g||+||g-\Sigma_m(g)||+||\Sigma_m(g-f)||<\epsilon
	\end{equation*}
	for any $m\ge M$. This shows that $(\Sigma_m(f))_{m\in\N}$ converges to $f$ in $C^*(\grp_E)$.
\end{proof}

\section{Orbit equivalence and pseudogroups}\label{sec: oe and pseudo}

In this section we introduce the notion of orbit equivalence of two graphs, which is a natural generalisation of  Matsumoto's 
continuous orbit equivalence for topological Markov shifts from \cite{Mat}. We also define the pseudogroup of a graph using Renault's pseudogroups 
associated to groupoids \cite{Ren2}, and then show that two graphs are orbit 
equivalent if and only if their pseudogroups are isomorphic. 

\begin{definition}\label{def: cont orb equiv}
Two graphs $E$ and $F$ are {\em orbit equivalent} if there 
exist a homeomorphism $h:\partial E\to \partial F$ and continuous functions 
$k_1,l_1:\partial E^{\ge 1}\to \N$ and $k'_1,l'_1:\partial F^{\ge 1}\to \N$ 
such that 
\begin{equation}\label{eq: con orb equiv}
\sigma_F^{k_1(x)}(h(\sigma_E(x)))=\sigma_F^{l_1(x)}(h(x))\quad\text{and}\quad\sigma_E^{k'_1(y)}(h^{-1}(\sigma_F(y)))=\sigma_E^{l'_1(y)}(h^{-1}(y)),
\end{equation}
for all $x\in\partial E^{\ge 1}, y\in \partial F^{\ge 1}$.
\end{definition}

\begin{example}\label{ex: coe but not conjugacy}
Consider the graphs
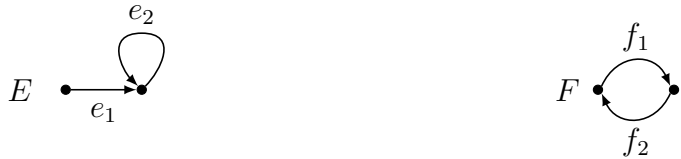
\[
\begin{tikzpicture}
    \def\vertex(#1) at (#2,#3){
        \node[inner sep=0pt, circle, fill=black] (#1) at (#2,#3)
        [draw] {.}; 
 }
    \vertex(11) at (0,0)
    
    \vertex(21) at (1,0)
    
    
    
    \vertex(51) at (7,0)
    
    \vertex(61) at (8,0)
    
    
    
\node at (-0.6,0) {$E$};
\node at (6.6,0) {$F$};


\node at (0.5,-0.3) {$e_1$};

\node at (1,1) {$e_2$};

\node at (7.5,0.7) {$f_1$};
\node at (7.5,-0.7) {$f_2$};

\draw[style=semithick, -latex] (11.east)--(21.west);

\draw[style=semithick, -latex] (51.north east)
.. controls (7.25,0.5) and (7.75,0.5).. (61.north west);
\draw[style=semithick, -latex] (61.south west)
.. controls (7.75,-0.5) and (7.25,-0.5).. (51.south east);

\draw[style=semithick, -latex] (21.north east)
        .. controls (1.25,0.25) and (1.5,0.75) ..
        (1,0.75)
        .. controls (0.5,0.75) and (0.75,0.25) ..
        (21.north west);
        
\end{tikzpicture}
\]
Then $\partial E=\{e_1e_2e_2\dots,e_2e_2\dots\}$ and $\partial F=\{f_1f_2f_1f_2\dots,f_2f_1f_2f_1\dots\}$. The map $h:\partial E\to \partial F$ given by
\[
h(e_1e_2e_2\dots)=f_1f_2f_1f_2\dots\quad\text{and}\quad h(e_2e_2\dots)=f_2f_1f_2f_1\dots
\]
is a homeomorphism. Consider $k_1,l_1:\partial E^{\ge 1}\to\N$ given by 
$k_1(e_1e_2e_2\dots)=1$ and $k_1(e_2e_2\dots)=0$, and 
$l_1(e_1e_2e_2\dots)=0=l_1(e_2e_2\dots)$. Then $k_1$ and $l_1$ are continuous, 
and we have $\sigma_F^{k_1(x)}(h(\sigma_E(x)))=\sigma_F^{l_1(x)}(h(x))$ for 
all $x\in\partial E^{\ge 1}$. Similarly the functions $k'_1,l'_1:\partial 
F^{\ge 1}\to\N$ given by $k'_1(f_1f_2f_1f_2\dots)=0$ and 
$k'_1(f_2f_1f_2f_1\dots)=1$, and  $l'_1(f_1f_2f_1f_2\dots)=1$ and 
$l'_1(f_2f_1f_2f_1\dots)=0$, are continuous and satisfy
\[
\sigma_E^{k'_1(y)}(h^{-1}(\sigma_F(y)))=\sigma_E^{l'_1(y)}(h^{-1}(y))\quad\text{for
 all $y\in\partial F^{\ge 1}$}.
\]
Hence $E$ and $F$ are orbit equivalent.
\end{example}

Sections \ref{sec: Statement of the main result} and \ref{sec: applications} contain further examples of orbit equivalent graphs.

In Section 3 of \cite{Ren2}, Renault constructs for each \'{e}tale groupoid 
$\grp$ a pseudogroup in the following way: Define a \emph{bisection} to be a 
subset $A$ of $\grp$ such that the restriction of the source map of $\grp$ to 
$A$ and the restriction of the range map of $\grp$ to $A$ both are injective. 
The set of all open bisections of $\grp$ forms an inverse semigroup 
$\mathcal{S}$ with product defined by $AB=\{\gamma\gamma':(\gamma,\gamma')\in 
(A\times B)\cap \grp^{(2)}\}$ (where $\grp^{(2)}$ denote the set of 
composable pairs of $\grp$), and the inverse of $A$ is defined to be the image 
of $A$ under the inverse map of $\grp$. Each $A\in\mathcal{S}$ defines a 
unique homeomorphism $\alpha_A:s(A)\to r(A)$ such that 
$\alpha(s(\gamma))=r(\gamma)$ for $\gamma\in A$. The set 
$\{\alpha_A:A\in\mathcal{S}\}$ of partial homeomorphisms on $\grp^0$ is the 
pseudogroup of $\grp$. 

When $E$ is a graph, then we call the pseudogroup of the \'{e}tale groupoid 
$\grp_E$ the \emph{pseudogroup of $E$} and denote it by $\mathcal{P}_E$. 

We will now give two alternative characterizations of the partial 
homeomorphisms of $\partial E$ that belong to $\mathcal{P}_E$.

\begin{prop} \label{prop:pseudo}
	Let $E$ be a graph, let $U$ and $V$ be open subsets of $\partial E$, and 
	let $\alpha:V\to U$ be a homeomorphism. Then the following are equivalent:
	\begin{itemize}
	\item[(1)] $\alpha\in\mathcal{P}_E$.
	\item[(2)] For all $x\in V$, there exist $m,n\in\N$ and an open subset 
	$V'$ such that $x\in V'\subseteq V$, and such that 
	$\sigma_E^m(x')=\sigma_E^n(\alpha(x'))$ for all $x'\in V'$.
	\item[(3)] There exist continuous functions $m,n:V\to\N$ such that 
	$\sigma_E^{m(x)}(x)=\sigma_E^{n(x)}(\alpha(x))$ for all $x\in V$.
	\end{itemize}
\end{prop}

\begin{proof}
	$(1)\Rightarrow (2)$: Suppose $\alpha\in\mathcal{P}_E$. Let 
	$A\in\mathcal{S}$ be such that $\alpha=\alpha_A$. Let $x\in V$. Then there 
	is a unique $\gamma\in A$ such that $s(\gamma)=x$, and then 
	$r(\gamma)=\alpha(x)$. Since $A$ is an open subset of $\grp_E$, there are 
	$m,n\in\N$, an open subset $U'$ of $\partial E^{\ge m}$ such that the 
	restriction of $\sigma_E^m$ to $U'$ is injective, and an open subset $V'$ 
	of $\partial E^{\ge n}$ such that the restriction of $\sigma_E^n$ to $V'$ 
	is injective and $\sigma_E^m(U')=\sigma_E^n(V')$, and such that $\gamma\in 
	Z(U',m,n,V')\subseteq A$. Then $x\in V'\subseteq V$ and 
	$\sigma_E^m(x')=\sigma_E^n(\alpha(x'))$ for all $x'\in V'$.
	
	$(2)\implies (3)$: Assume that for all $x\in V$, there exist $m,n\in\N$ 
	and an open subset $V'$ such that $x\in V'\subseteq V$, and such that 
	$\sigma_E^m(x')=\sigma_E^n(\alpha(x'))$ for all $x'\in V'$. According to 
	Lemma \ref{lem:open}, $V$ is the disjoint union of sets that are both 
	compact and open. Since $\partial E$ is locally compact, it follows that 
	there exists a family $\{V_i:i\in I\}$ of mutually disjoint compact and 
	open sets and a family $\{(m_i,n_i):i\in I\}$ of pairs of nonnegative 
	integers such that $V=\bigcup_{i\in I}V_i$ and 
	$\sigma_E^{m_i}(x)=\sigma_E^{n_i}(\alpha(x))$ for $x\in V_i$. Define 
	$m,n:V\to\N$ by setting $m(x)=m_i$ and $n(x)=n_i$ for $x\in V_i$. Then $m$ 
	and $n$ are continuous and $\sigma_E^{m(x)}(x)=\sigma_E^{n(x)}(\alpha(x))$ 
	for all $x\in V$.
	
	$(3)\implies (1)$: Assume that $m,n:V\to\N$ are continuous functions such 
	that $\sigma_E^{m(x)}(x)=\sigma_E^{n(x)}(\alpha(x))$ for all $x\in V$. 
	Then there exist for each $x\in V$ a compact and open subset $V_x$ such 
	that $x\in V_x\subseteq V$, $m(x')=m(x)$ and $n(x')=n(x)$ for all $x'\in 
	V_x$, the restriction of $\sigma_E^{n(x)}$ to $V_x$ is injective, and the 
	restriction of $\sigma_E^{m(x)}$ of $\alpha(V_x)$ is injective. According 
	to Lemma \ref{lem:open}, $V$ is the disjoint union of sets that are both 
	compact and open. It follows that there exists a family $\{V_i:i\in I\}$ 
	of mutually disjoint compact and open sets and a family $\{(m_i,n_i):i\in 
	I\}$ of pairs of nonnegative integers such that $V=\bigcup_{i\in I}V_i$, 
	$m(x)=m_i$ and $n(x)=n_i$ for all $x\in V_i$, the restriction of 
	$\sigma_E^{n_i}$ to $V_i$ is injective, and the restriction of 
	$\sigma_E^{m_i}$ of $\alpha(V_i)$ is injective. Let $A:=\bigcup_{i\in 
	I}Z(\alpha(V_i),m_i,n_iV_i)$. Then $A\in\mathcal{S}$ and 
	$\alpha=\alpha_A$, so $\alpha\in\mathcal{P}_E$.
\end{proof}

Suppose that $E$ and $F$ are two graphs and that there exists a homeomorphism 
$h:\partial E\to \partial F$. Let $U$ and $V$ be open subsets of $\partial E$ 
and let $\alpha:V\to U$ be a homeomorphism. We denote by 
$h\circ\mathcal{P}_E\circ h^{-1}:=\{h\circ\alpha\circ h^{-1}: 
\alpha\in\mathcal{P}_E\}$. We say that the pseudogroups of $E$ and $F$ are 
isomorphic if there is a homeomorphism $h:\partial E\to \partial F$ such that 
$h\circ\mathcal{P}_E\circ h^{-1}=\mathcal{P}_F$. We can now state the main 
result of this section.

\begin{prop}\label{prop:pseudogroups and orbit equivalence}
Let $E$ and $F$ be two graphs. Then $E$ and $F$ are orbit equivalent if and 
only if the pseudogroups of $E$ and $F$ are isomorphic.
\end{prop}

To prove this proposition we will use the following result.

\begin{lemma}\label{lem:orbit}
	Suppose two graphs $E$ and $F$ are orbit equivalent, $h:\partial E\to 
	\partial F$ is a homeomorphism and $k_1,l_1:\partial E^{\ge 1}\to \N$ and 
	$k'_1,l'_1:\partial F^{\ge 1}\to \N$ are continuous functions satisfying 
	\eqref{eq: con orb equiv}. Let $n\in\N$. Then there exist continuous 
	functions $k_n,l_n:\partial E^{\ge n}\to \N$ and $k'_n,l'_n:\partial 
	F^{\ge n}\to \N$ such that 
	\begin{equation}\label{eq: con orb equiv ext}
	\sigma_F^{k_n(x)}(h(\sigma_E^n(x)))=\sigma_F^{l_n(x)}(h(x))\quad\text{and}\quad\sigma_E^{k'_n(y)}(h^{-1}(\sigma_F^n(y)))=\sigma_E^{l'_n(y)}(h^{-1}(y)),
	\end{equation}
	for all $x\in\partial E^{\ge n}, y\in \partial F^{\ge n}$.
\end{lemma}

\begin{proof}
	There is nothing to prove for $n=0$ and $n=1$. We will prove the general 
	case by induction. Let $m\ge 1$ and suppose that the lemma holds for 
	$n=m$. Let $x\in\partial E^{\ge m+1}$. Then 
	\begin{equation*}
		\sigma_F^{k_1(\sigma_E^m(x))}(h(\sigma_E^{m+1}(x)))=\sigma_F^{l_1(\sigma_E^m(x))}(h(\sigma_E^m(x)))
	\end{equation*}
	and
	\begin{equation*}
		\sigma_E^{k_m(x)}(h(\sigma_E^m(x)))=\sigma^{l_m(x)}(h(x)). 
	\end{equation*}
	Let 
	\begin{align}
		k_{m+1}(x)&:=k_1(\sigma_E^m(x))+\max\{l_1(\sigma_E^m(x)),k_m(x)\}-l_1(\sigma_E^m(x))\label{eq:k}\\
		 l_{m+1}(x)&:=l_m(x)+\max\{l_1(\sigma_E^m(x)),k_m(x)\}-k_m(x)\label{eq:l}.
	\end{align}
	Then 
	\begin{equation*}
		\sigma_F^{k_{m+1}(x)}(h(\sigma_E^{m+1}(x)))=\sigma_F^{l_{m+1}(x)}(h(x)).
	\end{equation*}
	Since $k_1$, $l_1$, $k_m$, and $l_m$ are continuous, it follows that 
	$k_{m+1},l_{m+1}:\partial E^{\ge m+1}\to\N$ defined by \eqref{eq:k} and 
	\eqref{eq:l} are also continuous.
	
	Similarly, if we define $k'_{m+1},l'_{m+1}:\partial F^{\ge m+1}\to\N$ by 
	letting
	\begin{align*}
		k'_{m+1}(y)&:=k'_1(\sigma_F^m(y))+\max\{l'_1(\sigma_F^m(y)),k'_m(y)\}-l'_1(\sigma_F^m(y))\\
		 l'_{m+1}(y)&:=l_m(y)+\max\{l'_1(\sigma_F^m(y)),k'_m(y)\}-k'_m(y)
	\end{align*}
	for $y\in\partial F^{\ge m+1}$, then $k'_{m+1}$ and $l'_{m+1}$ are 
	continuous, and  
	\begin{equation*}
		\sigma_E^{k'_{m+1}(y)}(h^{-1}(\sigma_F^{m+1}(y)))=\sigma_E^{l'_{m+1}(y)}(h^{-1}(y))
	\end{equation*}
	for all $y\in\partial E^{\ge m+1}$. Thus, the lemma also holds for 
	$n=m+1$, and the general result holds by induction.
\end{proof}

\begin{proof}[Proof of Proposition~\ref{prop:pseudogroups and orbit equivalence}]
	Suppose $E$ and $F$ are orbit equivalent. Then there exists a 
	homeomorphism $h:\partial E\to \partial F$ and, for each $n\in\N$, there 
	exists 
	continuous functions $k_n, l_n:\partial E^{\ge n}\to\N$ satisfying the 
	first equation of \eqref{eq: con orb equiv ext}. Let $(\alpha:V\to 
	U)\in\mathcal{P}_E$, and let $m,n:V\to\N$ be continuous functions such 
	that $\sigma_E^{m(x)}(x)=\sigma_E^{n(x)}(\alpha(x))$ for all $x\in V$. Let 
	$y\in h(V)$. Then 
	\begin{align*}
		\sigma_F^{l_{n(h^{-1}(y))}(\alpha(h^{-1}(y)))}(h(\alpha(h^{-1}(y))))
		&=\sigma_F^{k_{n(h^{-1}(y))}(\alpha(h^{-1}(y)))}(h(\sigma_E^{n(h^{-1}(y))}(\alpha(h^{-1}(y)))))\\
		&=\sigma_F^{k_{n(h^{-1}(y))}(\alpha(h^{-1}(y)))}(h(\sigma_E^{m(h^{-1}(y))}(h^{-1}(y)))),
	\end{align*}
	and
	\begin{equation*}
		\sigma_F^{k_{m(h^{-1}(y))}(h^{-1}(y))}(h(\sigma_E^{m(h^{-1}(y))}(h^{-1}(y))))=
		\sigma_F^{l_{m(h^{-1}(y))}(h^{-1}(y))}(y).
	\end{equation*}
	So if we let 
	\begin{multline}\label{eq:m}
		m'(y):=l_{m(h^{-1}(y))}(h^{-1}(y))+\max\{k_{n(h^{-1}(y))}(\alpha(h^{-1}(y))),k_{m(h^{-1}(y))}(h^{-1}(y))\}\\-k_{m(h^{-1}(y))}(h^{-1}(y))
	\end{multline}
	and
	\begin{multline}\label{eq:n}
		n'(y):=l_{n(h^{-1}(y))}(\alpha(h^{-1}(y)))+\max\{k_{n(h^{-1}(y))}(\alpha(h^{-1}(y))),k_{m(h^{-1}(y))}(h^{-1}(y))\}\\-k_{n(h^{-1}(y))}(\alpha(h^{-1}(y))),
	\end{multline}
	then $\sigma_F^{m'(y)}(y)=\sigma^{n'(y)}(h(\alpha(h^{-1}(y))))$. Since 
	$h^{-1}$, $m$, $n$, and $\alpha$ are continuous, it follows that 
	$m',n':h(V)\to\N$ defined by \eqref{eq:m} and \eqref{eq:n} are also 
	continuous. Thus, it follows from Proposition \ref{prop:pseudo} that 
	$h\circ\alpha\circ h^{-1}\in\mathcal{P}_F$. A similar argument proves that 
	if $\alpha'\in\mathcal{P}_F$, then $h^{-1}\circ\alpha'\circ 
	h\in\mathcal{P}_E$. Thus $h\circ\mathcal{P}_E\circ h^{-1}=\mathcal{P}_F$ 
	and the pseudogroups of $E$ and $F$ are isomorphic.

Now suppose that $h:\partial E\to \partial F$ is a homeomorphism such that 
$h\circ\mathcal{P}_E\circ h^{-1}=\mathcal{P}_F$. Fix $e\in E^1$ and let 
$\alpha_e:=\sigma_E|_{Z(e)}$. Then $\alpha_e$ is a homeomorphism from $Z(e)$ to 
$\alpha_e(Z(e))$ and since $\alpha_e(x)=\sigma_E(x)$ for all $x\in Z(e)$, it 
follows from Proposition \ref{prop:pseudo} that $\alpha_e\in\mathcal{P}_E$. 
Thus $h\circ\alpha_e\circ h^{-1}\in\mathcal{P}_F$ by assumption. It follows 
from Proposition \ref{prop:pseudo} that there are continuous functions 
$m_e',n_e':h(Z(e))\to\N$ such that 
\begin{equation*}
\sigma_F^{n'_e(y)}(h(\alpha_e(h^{-1}(y))))=\sigma_F^{m'_e(y)}(y)\quad\text{for 
$y\in h(Z(e))$.}
\end{equation*}
Define functions $k_1,l_1:\partial E^{\ge 1}\to\N$ by $k_1(x)=n'_{x_1}(h(x))$ 
and $l_1(x)=m'_{x_1}(h(x))$, which are continuous because the $Z(e)$ are 
pairwise-disjoint compact open sets covering $\partial E^{\ge 1}$. Then for 
each $x=x_1x_2\dots\in \partial E$ we have
\begin{equation*}
	\sigma_F^{l_1(x)}(h(x))=\sigma_F^{m'_{x_1}(h(x))}(h(x))=\sigma_F^{n'_{x_1}(h(x))}(h(\alpha_{x_1}(x)))=\sigma_F^{k_1(x)}(h(\sigma_E(x))).
\end{equation*}
Hence $k_1$ and $l_1$ satisfy the first equation from \eqref{eq: con orb 
equiv}. A similar argument gets the second equation from \eqref{eq: con orb 
equiv}. Thus $E$ and $F$ are orbit equivalent.
\end{proof}

\section{The extended Weyl groupoid of $(C^*(E),\mathcal{D}(E))$}\label{sec: 
ext Weyl}

Proposition~\ref{prop:groupoid} says that the pair $(C^*(E),\mathcal{D}(E))$ 
is an invariant of $\grp_E$, in the sense that if $E$ and $F$ are two graphs 
such that $\grp_E$ and $\grp_F$ are isomorphic as topological groupoids, then 
there is an isomorphism from $C^*(E)$ to $C^*(F)$ which maps $\mathcal{D}(E)$ 
onto $\mathcal{D}(F)$. In this section we show that $\grp_E$ is an invariant 
of $(C^*(E),\mathcal{D}(E))$, in the sense that if there is an isomorphism 
from $C^*(E)$ to $C^*(F)$ which maps $\mathcal{D}(E)$ onto $\mathcal{D}(F)$, 
then $\grp_E$ and $\grp_F$ are isomorphic as topological groupoids.

To prove this result we build a groupoid from $(C^*(E),\DD(E))$ that we call the 
extended Weyl groupoid, which generalises Renault's Weyl groupoid construction 
from \cite{Ren2} applied to $(C^*(E),\DD(E))$. Recall from \cite{Ren2} that 
Weyl groupoids are associated to pairs $(A,B)$ consisting of a $C^*$-algebra 
$A$ and an abelian $C^*$-subalgebra $B$ which contains an approximate unit of 
$A$. The Weyl groupoid construction has the property that if $\grp$ is a 
topologically principal \'etale Hausdorff locally compact second countable 
groupoid and $A=C^*_{\operatorname{red}}(\grp)$ and $B=C_0(\grp^0)$, then the 
associated Weyl groupoid is isomorphic to $\grp$ as a topological groupoid. We 
will modify Renault's construction for pairs $(C^*(E),\mathcal{D}(E))$ to 
obtain a groupoid $\grp_{(C^*(E),\mathcal{D}(E))}$ such that $\grp_{(C^*(E),\mathcal{D}(E))}$ and $\grp_E$ are isomorphic as topological groupoids, even when $\grp_E$ is not topologically 
principal. We will then show that if $E$ and $F$ are two graphs such that 
there is an isomorphism from $C^*(E)$ to $C^*(F)$ which maps $\mathcal{D}(E)$ 
onto $\mathcal{D}(F)$, then $\grp_{(C^*(E),\mathcal{D}(E))}$ and 
$\grp_{(C^*(F),\mathcal{D}(F))}$, and thus $\grp_E$ and $\grp_F$ are 
isomorphic as topological groupoids.

As in \cite{Ren2} (and originally defined in \cite{Kum2}), we define the 
\emph{normaliser} of $\mathcal{D}(E)$ to be the set
\begin{equation*}
	N(\mathcal{D}(E)):=\{n\in C^*(E): ndn^*, n^*dn\in\mathcal{D}(E)\text{ for 
	all }d\in\mathcal{D}(E)\}.
\end{equation*} 
According to \cite[Lemma 4.6]{Ren2}, $nn^*,n^*n\in\mathcal{D}(E)$ for $n\in 
N(\mathcal{D}(E))$. Recalling the definition of $h_E$ given in \eqref{he_map}, 
for $n\in N(\mathcal{D}(E))$, we let $\dom(n):=\{x\in\partial 
E:h_E(x)(n^*n)>0\}$ and $\ran(n):=\{x\in\partial E:h_E(x)(nn^*)>0\}$. It 
follows from \cite[Proposition 4.7]{Ren2} that, for $n\in N(\mathcal{D}(E))$, 
there is a unique homeomorphism $\alpha_n:\dom(n)\to\ran(n)$ such that, for 
all $d\in\mathcal{D}(E)$,
\begin{equation}\label{eq: defining prop of alphas}
	h_E(x)(n^*dn)=h_E(\alpha_n(x))(d)h_E(x)(n^*n).
\end{equation}
From \cite[Lemma~4.10]{Ren2} we also know that $\alpha_{n^*}=\alpha_n^{-1}$ 
and $\alpha_{mn}=\alpha_m\circ\alpha_n$ for each $m,n\in N(\DD(E))$. 

The following lemma gives an insight into how the homeomorphisms $\alpha_n$ 
work. We 
collect further properties of these homeomorphisms in Lemma~\ref{lem: 
properties of alpha n}.

\begin{lemma}\label{lem: alpha sub spanning elt }
Let $E$ be a graph. For each $\mu,\nu\in E^*$ with $r(\mu)=r(\nu)$ we have 
$s_{\mu}s_\nu^*\in N(\DD(E))$ with
\[
\dom(s_\mu s_\nu^*)=Z(\nu),\,\ran(s_\mu 
s_\nu^*)=Z(\mu)\text{ and }\alpha_{s_\mu s_\nu^*}(\nu z)=\mu z\text{ 
for all $z\in 
r(\nu)\partial E$.}
\]
\end{lemma}

\begin{proof}
Let $\mu,\nu\in E^*$ with $r(\mu)=r(\nu)$. For each $\lambda\in E^*$ we have
\begin{equation}\label{eq: key s mu s nu star normaliser}
(s_\mu s_\nu^*)^*s_\lambda s_\lambda^*(s_\mu s_\nu ^*) =
\begin{cases}
s_\nu s_\nu^* & \text{if $\mu=\lambda\mu'$}\\
s_{\nu\lambda'}s_{\nu\lambda'}^* & \text{if $\lambda=\mu\lambda'$}\\
0 & \text{otherwise.}
\end{cases}
\end{equation}
So $(s_\mu s_\nu^*)^*s_\lambda s_\lambda^*(s_\mu s_\nu ^*)\in\DD(E)$, and it 
follows that $(s_\mu s_\nu^*)^*d(s_\mu s_\nu ^*)\in\DD(E)$ for all 
$d\in\DD(E)$. A similar argument shows that $(s_\mu s_nu^*)d(s_\mu s_\nu 
^*)^*\in\DD(E)$ for all $d\in\DD(E)$, and hence $s_\mu s_\nu^*\in N(D(E))$. We 
have
\[
h_E(x)((s_\mu s_{\nu}^*)^*s_\mu s_\nu^*)=h_E(x)(s_\nu s_\nu^*)=
\begin{cases}
1 & \text{if $x\in Z(\nu)$}\\
0 & \text{if $x\not\in Z(\nu)$},
\end{cases}
\]
and hence $\dom(s_\mu s_\nu^*)=Z(\nu)$. A similar calculation gives 
$\ran(s_\mu s_\nu^*)=Z(\mu)$. 

Now suppose $z\in 
r(\nu)\partial E$. We use (\ref{eq: defining prop of alphas}) and (\ref{eq: 
key s mu s nu star normaliser}) to get
\begin{align*}
h_E(\alpha_{s_\mu s_\nu^*}(\nu z))(s_\lambda s_\lambda^*) &= h_E(\nu 
z)\big((s_\mu 
s_\nu^*)^*s_\lambda s_\lambda^*(s_\mu s_\nu ^*)\big)h_E(\nu z)\big((s_\mu 
s_{\nu}^*)^*s_\mu s_\nu^*\big)\\
&= 
\begin{cases}
h_E(\nu z)(s_\nu s_\nu^*) & \text{if $\mu=\lambda\mu'$}\\
h_E(\nu z)(s_{\nu\lambda'} s_{\nu\lambda'}^*) & \text{if 
$\lambda=\mu\lambda'$}\\
0 & \text{otherwise}
\end{cases}\\
&= 
\begin{cases}
1 & \text{if $\mu z\in Z(\lambda)$}\\
0 & \text{otherwise}
\end{cases}\\
&= h_E(\mu z)(s_\lambda s_\lambda^*).
\end{align*}
It follows that $h_E(\alpha_{s_\mu s_\nu^*}(\nu z))=h_E(\mu z)$, and hence 
$\alpha_{s_\mu s_\nu^*}(\nu z)=\mu z$.
\end{proof}

Denote by $\partial E_\is$ the set of isolated points in $\partial E$. Notice 
that $x\in\partial E$ belongs to $\partial E_\is$ if and only if the 
characteristic function $1_{\{x\}}$ belongs to $C_0(\partial E)$. For 
$x\in\partial E_\is$, we let $p_x$ denote the unique element of 
$\mathcal{D}(E)$ satisfying that $h_E(y)(p_x)$ is 1 if $y=x$ and zero 
otherwise.

\begin{lemma}\label{lem: properties of alpha n}
Let $E$ be a graph, $n\in N(\DD(E))$ and $x\in\partial E_\is\cap\dom(n)$. Then
\begin{enumerate}
\item[(a)] $np_xn^*=h_E(x)(n^*n)p_{\alpha_n(x)}$,
\item[(b)] $n^*p_{\alpha_n(x)}n=h_E(x)(n^*n)p_x$, and 
\item[(c)] $np_x=p_{\alpha_n(x)}n$.
\end{enumerate}
\end{lemma}

\begin{proof}
We use Equation~(\ref{eq: defining prop of alphas}) with 
$d=nn^*$ 
to get
\[
h_E(x)(n^*n)^2=h_E(x)(n^*nn^*n)=h_E(\alpha_n(x))(nn^*)h_E(x)(n^*n),
\]
which implies that 
\begin{equation}\label{eq: hEs and alphas}
h_E(\alpha_n(x))(nn^*)=h_E(x)(n^*n). 
\end{equation}
Note that this is a 
positive number because $x\in\dom(n)$. For (a) we again use 
(\ref{eq: defining prop of alphas}) to get
\begin{align*}
h_E(y)\big(\bigl(h_E(\alpha_n(x))(nn^*)\bigr)^{-1}np_xn^*\big) &= 
\bigl(h_E(\alpha_n(x))(nn^*)\bigr)^{-1}h_E(y)(np_xn^*)\\
&= \bigl(h_E(\alpha_n(x))(nn^*)\bigr)^{-1}h_E(\alpha_{n^*}(y))(p_x)h_E(y)(nn^*)\\
&= 
\begin{cases}
1 & \text{if $y=\alpha_n(x)$}\\
0 & \text{otherwise}.
\end{cases}
\end{align*}
By the defining property of $p_{\alpha_n(x)}$ we now have 
$p_{\alpha_n(x)}=\bigl(h_E(\alpha_n(x))(nn^*)\bigr)^{-1}np_xn^*$. Using (\ref{eq: hEs and 
alphas}) gives $np_xn^*=h_E(x)(n^*n)p_{\alpha_n(x)}$, which is (a). 
Identity (b) follows from (a) by replacing $n$ with $n^*$ and $x$ with 
$\alpha_n(x)$ and then use (\ref{eq: hEs and 
alphas}).

To prove (c) we first notice that
\begin{align*}
h_E(y)\big((h_E(x)(n^*n))^{-1}n^*np_x\big)=(h_E(x)(n^*n))^{-1}h_E(y)(n^*n)h_E(y)(p_x)=
\begin{cases}
1 & \text{if $x=y$}\\
0 & \text{if $x\not=y$.}
\end{cases}
\end{align*} 
Hence by the defining property of $p_x$ we have
\begin{equation}\label{eq: n*n and hE}
n^*np_x=h_E(x)(n^*n)p_x.
\end{equation} 
We now use (\ref{eq: n*n and hE}) and (a) to get (c): 
\[
np_x = n((h_E(x)(n^*n))^{-1}n^*np_x)=(h_E(x)(n^*n))^{-1}np_x 
n^*n=p_{\alpha_n(x)}n.\qedhere
\] 
\end{proof}

\begin{lemma}\label{lem:corner}
	Suppose $x\in\partial E_\is$. If $x$ is not eventually periodic, then 
	$p_xC^*(E)p_x=p_x\mathcal{D}(E)p_x=\C p_x$. If $x=\mu\eta\eta\eta\cdots$ 
	for some $\mu\in E^*$ and some simple loop $\eta\in E^*$ with 
	$s(\eta)=r(\mu)$, then $p_xC^*(E)p_x$ is isomorphic to $C(\T)$ by the 
	isomorphism mapping $p_xs_\mu s_\eta s_\mu^*p_x$ to the identity function 
	on $\T$, and $p_x\mathcal{D}(E)p_x=\C p_x$. 
\end{lemma}

\begin{proof}
	Let $(\grp_E)_x^x$ denote the isotropy group 
	$\{\gamma\in\grp:s(\gamma)=r(\gamma)=x\}$ of $(x,0,x)$. Assume that $x$ is 
	not eventually periodic. Then $(\grp_E)_x^x=\{(x,0,x)\}$. Proposition 
	\ref{prop:groupoid} implies that there is an isomorphism from 
	$p_xC^*(E)p_x$ to $C^*((\grp_E)_x^x)$, and consequently  
	$p_xC^*(E)p_x=p_x\mathcal{D}(E)p_x=\C p_x$, completing the first assertion 
	in the lemma.
	
	Assume then that $x=\mu\eta\eta\eta\cdots$ for some $\mu\in E^*$ and some 
	simple loop $\eta\in E^*$ with $s(\eta)=r(\mu)$. We then have that 
	$(\grp_E)_x^x=\{(x,k|\eta|,x):k\in\Z\}$. Now Proposition 
	\ref{prop:groupoid} implies that there is an isomorphism from $p_xC^*(E)p_x$ to 
	$C(\T)$ which maps $p_xs_\mu s_\eta s_\mu^*p_x$ to the identity function 
	on $\T$, and that $p_x\mathcal{D}(E)p_x=\C p_x$.
\end{proof}

The extended Weyl groupoid associated to $(C^*(E),\DD(E))$ is built using an 
equivalence relation defined on pairs of normalisers and boundary paths. For 
isolated boundary paths $x$ the equivalence relation is defined using a 
unitary in the corner of $C^*(E)$ determined by $p_x$.

\begin{lemma}\label{lem: props of the unitary}
Let $E$ be a graph. For $x\in\partial 
E_\is$, $n_1,n_2\in N(\mathcal{D}(E))$, $x\in\dom(n_1)\cap\dom(n_2)$, and 
$\alpha_{n_1}(x)=\alpha_{n_2}(x)$, we denote
\[
U_{(x,n_1,n_2)}:=(h_E(x)(n_1^*n_1n_2^*n_2))^{-1/2}p_xn_1^*n_2p_x.
\]
Then
\begin{enumerate}
\item[(1)] 
$U_{(x,n_1,n_2)}U_{(x,n_1,n_2)}^*=U_{(x,n_1,n_2)}^*U_{(x,n_1,n_2)}=p_x$, and
\item[(2)] $U_{(x,n_1,n_2)}^*=U_{(x,n_2,n_1)}$.
\end{enumerate}
Moreover, if $n_3\in N(\mathcal{D}(E))$, $x\in\dom(n_3)$, and 
$\alpha_{n_3}(x)=\alpha_{n_1}(x)=\alpha_{n_2}(x)$, then
\begin{enumerate}  
\item[(3)] $U_{(x,n_1,n_2)}U_{(x,n_2,n_3)}=U_{(x,n_1,n_3)}$. 
\end{enumerate}
\end{lemma}

\begin{proof}
Suppose that $x\in\partial 
E_\is$, $n_1,n_2\in N(\mathcal{D}(E))$, $x\in\dom(n_1)\cap\dom(n_2)$, and 
$\alpha_{n_1}(x)=\alpha_{n_2}(x)$. First note that since 
$x\in\dom(n_1)\cap\dom(n_2)$, we have $h_E(n_1^*n_1),h_E(n_2^*n_2)>0$, and the 
formula for $U_{(x,n_1,n_2)}$ makes sense. We now claim that
\begin{equation}\label{eq: key eq for U a unitary}
p_xn_1^*n_2p_xn_2^*n_1p_x=h_E(x)(n_1^*n_1n_2^*n_2)p_x.
\end{equation}
To see this, we apply identities (a) and (b) of Lemma~\ref{lem: properties of 
alpha n} to get
\begin{align*}
p_xn_1^*n_2p_xn_2^*n_1p_x &= p_xn_1^*(n_2p_xn_2^*)n_1p_x\\
&= 
p_xn_1^*\big(h_E(x)(n_2^*n_2)p_{\alpha_{n_2}(x)}\big)n_1p_x\\
&= h_E(x)(n_2^*n_2)p_xn_1^*p_{\alpha_{n_1}(x)}n_1p_x\\
&= 
h_E(x)(n_1^*n_1)h_E(x)(n_2^*n_2)p_x\\
&= h_E(x)(n_1^*n_1n_2^*n_2)p_x.
\end{align*}
We now use (\ref{eq: key eq for U 
a unitary}) to get
\[
U_{(x,n_1,n_2)}U_{(x,n_1,n_2)}^*=(h_E(n_1^*n_1n_2^*n_2))^{-1}
p_xn_1^*n_2p_xn_2^*n_1p_x=p_x
\]
Similarly, and using that $p_xC^*(E)p_x$ is commutative, we have
\begin{align*}
U_{(x,n_1,n_2)}^*U_{(x,n_1,n_2)}
&=(h_E(n_1^*n_1n_2^*n_2))^{-1}p_xn_2^*n_1p_xn_1^*n_2p_x\\
&=(h_E(n_1^*n_1n_2^*n_2))^{-1}p_xn_1^*n_2p_xn_2^*n_1p_x\\
&=p_x.
\end{align*}
So (1) holds.

Identity (2) holds because $n_1^*n_1,n_2^*n_2\in\DD(E)$, and hence
\begin{align*}
U_{(x,n_1,n_2)}^*:=(h_E(x)(n_1^*n_1n_2^*n_2))^{-1/2}p_xn_2^*n_1p_x&=
(h_E(x)(n_2^*n_2n_1^*n_1))^{-1/2}p_xn_2^*n_1p_x\\
&=U_{(x,n_2,n_1)}.
\end{align*}

We use identities (a) and (c) of Lemma~\ref{lem: properties of alpha n} to get
\begin{align*}
U_{(x,n_1,n_2)}U_{(x,n_2,n_3)}&=(h_E(x)(n_1^*n_1n_2^*n_2)h_E(x)(n_2^*n_2n_3^*n_3))^{-1/2}
p_xn_1^*n_2p_xn_2^*n_3p_x\\
&= (h_E(x)(n_2^*n_2))^{-1}(h_E(x)(n_1^*n_1n_3^*n_3))^{-1/2}p_xn_1^*
\big(h_E(x)(n_2^*n_2)
p_{\alpha_{n_2}(x)}\big)n_3p_x\\
&=(h_E(n_1^*n_1n_3^*n_3))^{-1/2}p_xn_1^*
p_{\alpha_{n_3}(x)}n_3p_x\\
&= (h_E(n_1^*n_1n_3^*n_3))^{-1/2}p_xn_1^*n_3p_x\\
&= U_{(x,n_1,n_3)}.
\end{align*}
So (3) holds. 
\end{proof}

\begin{notation}\label{notation: lambda}
For $x$, $n_1$ and $n_2$ as in Lemma~\ref{lem: props of the unitary} we 
denote 
\[
\lambda_{(x,n_1,n_2)}:=h_E(x)(n_1^*n_1n_2^*n_2).
\]
So $U_{(x,n_1,n_2)}=\lambda_{(x,n_1,n_2)}^{-1/2}p_xn_1^*n_2p_x$. It follows from identity (1) of Lemma \ref{lem: props of the unitary} that $U_{(x,n_1,n_2)}$ is a unitary element of $p_xC^*(E)p_x$. We denote by 
$[U_{(x,n_1,n_2)}]_1$ the class of $U_{(x,n_1,n_2)}$ in 
$K_1(p_xC^*(E)p_x)$.
\end{notation}

\begin{prop}\label{prop: equiv relation}
Let $E$ be a graph. For each $x_1,x_2\in\partial E$ and $n_1,n_2\in 
N(\mathcal{D}(E))$ such that $x_1\in\dom(n_1)$ and $x_2\in\dom(n_2)$ we write 
$(n_1,x_1)\sim (n_2,x_2)$ if either
\begin{enumerate}
\item[(a)] $x_1=x_2\in\partial E_\is$, 
$\alpha_{n_1}(x_1)=\alpha_{n_2}(x_2)$, and $[U_{(x_1,n_1,n_2)}]_1=0$; or
\item[(b)] $x_1=x_2\notin\partial E_\is$ and there is an open set $V$ such 
that $x_1\in 
V\subseteq\dom(n_1)\cap\dom(n_2)$ and $\alpha_{n_1}(y)=\alpha_{n_2}(y)$ for 
all $y\in V$. 
\end{enumerate}
Then $\sim$ is an equivalence relation on $\{(n,x):n\in 
N(\mathcal{D}(E)),\ x\in\dom(n)\}$.
\end{prop}

\begin{proof}
The only nontrivial parts to prove are that $\sim$ is symmetric and transitive 
when the boundary paths are isolated points. Suppose $(n_1,x_1)\sim (n_2,x_2)$ 
with $x:=x_1=x_2\in\partial E_\is$. We know from Lemma~\ref{lem: props of the 
unitary}(2) that $U_{(x,n_2,n_1)}=U_{(x,n_1,n_2)}^*$. So
\[
[U_{(x,n_1,n_2)}]_1=0\Longrightarrow 
[U_{(x,n_2,n_1)}]_1=[U_{(x,n_1,n_2)}^*]_1=0,
\]
and hence $(n_2,x_2)\sim (n_1,x_1)$.

For transitivity, suppose $(n_1,x_1)\sim (n_2,x_2)$ and $(n_2,x_2)\sim 
(n_3,x_3)$ with 
$x:=x_1=x_2=x_3\in\partial E_\is$. We know from Lemma~\ref{lem: props of the 
unitary}(2) that $U_{(x,n_1,n_2)}U_{(x,n_2,n_3)}=U_{(x,n_1,n_3)}$. So
\[
[U_{((x,n_1,n_2)}]_1=0=[U_{(x,n_2,n_3)}]_1\Longrightarrow 
[U_{(x,n_1,n_3)}]_1=[U_{(x,n_1,n_2)}]_1[U_{(x,n_2,n_3)}]_1=0,
\]
and hence $(n_1,x_1)\sim (n_3,x_3)$.
\end{proof}

\begin{prop}\label{prop: the groupoid}
Let $E$ be a graph, and $\sim$ the equivalence relation on $\{(n,x):n\in 
N(\mathcal{D}(E)),\ x\in\dom(n)\}$ from Proposition~\ref{prop: equiv 
relation}. Denote the collection of equivalence classes by 
$\grp_{(C^*(E),\mathcal{D}(E))}$. Define a partially-defined product on 
$\grp_{(C^*(E),\mathcal{D}(E))}$ by 
\[
[(n_1,x_1)][(n_2,x_2)]:=[(n_1n_2, x_2)]\quad\text{ if $\alpha_{n_2}(x_2)=x_1$},
\]
and undefined otherwise. Define an inverse map by 
$[(n,x)]^{-1}:=[(n^*,\alpha_n(x))]$. Then these operations make $\grp_{(C^*(E),\mathcal{D}(E))}$ into a
groupoid. 
\end{prop}

\begin{proof}
We only check that composition and inversion are well-defined. That 
composition is associative and every element is composable with its inverse 
(in either direction) is left to the reader. To see that composition is 
well-defined, suppose 
$[(n_1,x_1)]=[(n_1',x_1')]$ and $[(n_2,x_2)]=[(n_2',x_2')]$ with $[(n_1,x_1)]$ 
and $[(n_2,x_2)]$ composable. We need to show that $[(n_1',x_1')]$ and 
$[(n_2',x_2')]$ are also composable with 
\begin{equation}\label{comp_grp}
[(n_1n_2,x_2)]=[(n_1'n_2',x_2')].
\end{equation}
We immediately know that $x_1=x_1'$, $x_2=x_2'$, $x_2=\alpha_{n_2}^{-1}(x_1)$, 
$\alpha_{n_1}(x_1)=\alpha_{n_1'}(x_1')$, and 
$\alpha_{n_2}(x_2)=\alpha_{n_2'}(x_2')$. This gives
\[
\alpha_{n_2'}^{-1}(x_1')=\alpha_{n_2'}^{-1}(x_1)=\alpha_{n_2'}^{-1}(\alpha_{n_2}(x_2))=\alpha_{n_2'}^{-1}(\alpha_{n_2'}(x_2'))=x_2'.
\]
So $\alpha_{n_2'}(x_2')=x_1'$, and hence $[(n_1',x_1')]$ and $[(n_2',x_2')]$ 
are composable.

To see that \eqref{comp_grp} holds we have two cases: \\
\emph{Case 1}: Suppose $x_1 \notin \partial E_\is$.  Then $x_1'=x_1 \notin 
\partial E_\is$, $x_2=\alpha_{n_2}^{-1}(x_1) \notin \partial E_\is$, and $x_2' 
= \alpha_{n_2'}^{-1}(x_1') \notin \partial E_\is$. We also know there exists 
an open set $V_1$ such that $x_1 \in V_1 \subseteq \dom(n_1) \cap \dom(n_1')$ 
with $\alpha_{n_1}|_{V_1}=\alpha_{n_1'}|_{V_1}$, and an open set $V_2$ such 
that $x_2 \in V_2 \subseteq \dom(n_2) \cap \dom(n_2')$ with 
$\alpha_{n_2}|_{V_2}=\alpha_{n_2'}|_{V_2}$. Let $V:=V_2 \cap 
\alpha_{n_2}^{-1}(V_1)$, which is an open set containing $x_2$. We claim that 
\[
V\subseteq \dom(n_1n_2)\cap \dom(n_1'n_2').
\]
To see this, let $x 
\in V$. Then using \eqref{eq: defining prop of alphas} we have
\[
h_E(x)((n_1n_2)^*n_1 n_2)=h_E(\alpha_{n_2}(x))(n_1^*n_1)h_E(x)(n_2^*n_2),
\]
which is positive because $\alpha_{n_2}(x)\in\dom(n_1)$ and $x\in\dom(n_2)$. 
So $V\subseteq\dom(n_1n_2)$. A similar argument gives 
$V\subseteq\dom(n_1'n_2')$, and so the claim holds. For each $x\in V$ we have 
$\alpha_{n_2}(x)=\alpha_{n_2'}(x)\in V_1$, which means
\[
\alpha_{n_1n_2}(x)=\alpha_{n_1}(\alpha_{n_2}(x))=\alpha_{n_1'}(\alpha_{n_2'}(x))
=\alpha_{n_1'n_2'}(x).
\]
So $\alpha_{n_1n_2}|_V=\alpha_{n_1'n_2'}|_V$. Hence 
$(n_1n_2,x_2)\sim(n_1'n_2',x_2')$, and (\ref{comp_grp}) holds 
in this case.

\emph{Case 2}: Suppose $x_1 \in \partial E_\is$. Then $x_1'=x_1 \in \partial 
E_\is$, $x_2=\alpha_{n_2}^{-1}(x_1) \in \partial E_\is$, and 
$x_2'=\alpha_{n_2'}^{-1}(x_1') \in \partial E_\is$. We also have 
$\alpha_{n_1}(x_1)=\alpha_{n_1'}(x_1')$, 
$\alpha_{n_2}(x_2)=\alpha_{n_2'}(x_2')$, and hence
\[
\alpha_{n_1 
n_2}(x_2)=\alpha_{n_1}(\alpha_{n_2'}(x_2'))=\alpha_{n_1}(x_1')=\alpha_{n_1'}(x_1')=\alpha_{n_1'}(\alpha_{n_2'}(x_2'))=\alpha_{n_1'n_2'}(x_2').
\]
To get $(n_1n_2,x_2)\sim(n_1'n_2',x_2')$ in this case it now suffices to show 
that
$[U_{(x_2,n_1n_2,n_1'n_2')}]_1=0$.
We use that $\alpha_{n_2}(x_2)=x_1$ and $\alpha_{n_2'}(x_2')=x_1'=x_1$ and 
apply Lemma~\ref{lem: properties of alpha n}(c) twice to get
\[
p_{x_2}n_2^*n_1^*n_1'n_2'p_{x_2} = (n_2p_{x_2})^*n_1^*n_1'(n_2'p_{x_2})=
(p_{\alpha_{n_2}(x_2)}n_2)^*n_1^*n_1'p_{\alpha_{n_2'}(x_2)}n_2'= 
n_2^*p_{x_1}n_1^*n_1'p_{x_1}n_2'.
\]
Now we can write
\begin{align*}
U_{(x_2,n_1n_2,n_1'n_2')} &= 
\lambda_{(x_2,n_1n_2,n_1'n_2')}^{-1/2}p_{x_2}n_2^*n_1^*n_1'n_2'p_{x_2}\\
&= 
\lambda_{(x_2,n_1n_2,n_1'n_2')}^{-1/2}n_2^*p_{x_1}n_1^*n_1'p_{x_1}n_2'\\
&= \lambda_{(x_2,n_1n_2,n_1'n_2')}^{-1/2}\lambda_{(x_1,n_1,n_1')}^{1/2}
n_2^*\big(\lambda_{(x_1,n_1,n_1')}^{-1/2}p_{x_1}n_1^*n_1'p_{x_1}\big)
n_2'\\
&=\lambda_{(x_2,n_1n_2,n_1'n_2')}^{-1/2}\lambda_{(x_1,n_1,n_1')}^{1/2}
n_2^*U_{(x_1,n_1,n_1')}n_2'
\end{align*} 
Since $(n_1,x_1)\sim (n_1',x_1')$ implies that $U_{(x_1,n_1,n_1')}$ is 
homotopic to $p_{x_1}$, we see that $U_{(x_2,n_1n_2,n_1'n_2')}$ is homotopic to
\[
\lambda_{(x_2,n_1n_2,n_1'n_2')}^{-1/2}\lambda_{(x_1,n_1,n_1')}^{1/2}
n_2^*p_{x_1}n_2'.
\]
We use Lemma~\ref{lem: properties of alpha n}(c) to get
\[
n_2^*p_{x_1}n_2'n_2^*p_{\alpha_{n_2}(x_2)}p_{\alpha_{n_2'}(x_2')}n_2'
=p_{x_2}n_2^*n_2'p_{x_2}.
\]
Hence $U_{(x_2,n_1n_2,n_1'n_2')}$ is homotopic to
\begin{align*}
\lambda_{(x_2,n_1n_2,n_1'n_2')}^{-1/2}\lambda_{(x_1,n_1,n_1')}^{1/2}
n_2^*p_{x_1}n_2'&=\lambda_{(x_2,n_1n_2,n_1'n_2')}^{-1/2}
\lambda_{(x_1,n_1,n_1')}^{1/2}
p_{x_2}n_2^*n_2'p_{x_2}\\
&=\lambda_{(x_2,n_1n_2,n_1'n_2')}^{-1/2}
\lambda_{(x_1,n_1,n_1')}^{1/2}\lambda_{(x_2,n_2,n_2')}^{1/2}
U_{(x_2,n_2,n_2')}.
\end{align*}
But $(n_2,x_2)\sim (n_2',x_2')$ implies that $U_{(x_2,n_2,n_2')}$ is 
homotopic to $p_{x_2}$, and hence $U_{(x_2,n_1n_2,n_1'n_2')}$ is homotopic to 
$p_{x_2}$. This says that $[U_{(x_2,n_1n_2,n_1'n_2')}]_1=0$, as desired.

This complete the proof that composition is well-defined. To see that 
inversion is well-defined, suppose $[(n_1,x_1)]=[(n_1',x_1')]$. We need to 
show 
that $[(n_1^*,\alpha_{n_1}(x_1))]=[((n_1')^*,\alpha_{n_1'}(x_1')]$. We again 
have two cases.

{\em Case 1:} Suppose that $x_1=x_1'\notin\partial E_\is$. We know that there 
is open $V$ such that $x_1\in V\subseteq \dom(n_1)\cap\dom(n_1')$ and 
$\alpha_{n_1}|_V=\alpha_{n_1'}|_V$. A straightforward argument shows that the 
open set $V':=\alpha_{n_1}(V)$ satisfies $\alpha_{n_1}(x_1)\in V'\subseteq 
\dom(n_1^*)\cap \dom((n_1')^*)$ and
$\alpha_{n_1^*}|_{V'}=\alpha_{(n_1')^*}|_{V'}$. So 
$[(n_1^*,\alpha_{n_1}(x_1))]=[((n_1')^*,\alpha_{n_1'}(x_1')]$ in this case.

{\em Case 2:} Suppose that $x_1=x_1'\in\partial E_\is$. We have to show that 
$[U_{(\alpha_{n_1}(x_1),n_1^*,(n_1')^*)}]_1=0$. We use Lemma~\ref{lem: 
properties of alpha n}(c) to get
\begin{align*}
U_{(\alpha_{n_1}(x_1),n_1^*,(n_1')^*)} 
&=\lambda_{(\alpha_{n_1}(x_1),n_1^*,(n_1')^*)}^{-1/2}
p_{\alpha_{n_1}(x_1)}n_1(n_1')^*p_{\alpha_{n_1}(x_1)}\\
&= \lambda_{(\alpha_{n_1}(x_1),n_1^*,(n_1')^*)}^{-1/2}n_1p_{x_1}(n_1')^*.
\end{align*}
Since $(n_1,x_1)\sim(n_1',x_1')$, we have $U_{(x_1,n_1,n_1')}$ homotopic to 
$p_{x_1}$. Hence $U_{(\alpha_{n_1}(x_1),n_1^*,(n_1')^*)}$ is homotopic to 
\[
\lambda_{(\alpha_{n_1}(x_1),n_1^*,(n_1')^*)}^{-1/2}n_1U_{(x_1,n_1,n_1')}(n_1')^*
=\lambda_{(\alpha_{n_1}(x_1),n_1^*,(n_1')^*)}^{-1/2}\lambda_{(x_1,n_1,n_1')}^{-1/2}
n_1p_{x_1}n_1^*n_1'p_{x_1}(n_1')^*.
\]
Now, using (\ref{eq: hEs and alphas}) we have
\[
\lambda_{(\alpha_{n_1}(x_1),n_1^*,(n_1')^*)}^{-1/2}\lambda_{(x_1,n_1,n_1')}^{-1/2}
= h_E(x_1)(n_1^*n_1)^{-1}h_E(x_1)((n_1')^*n_1')^{-1}.
\]
So $U_{(\alpha_{n_1}(x_1),n_1^*,(n_1')^*)}$ is homotopic to 
\begin{align*}
h_E(x_1)(n_1^*n_1)^{-1}&h_E(x_1)((n_1')^*n_1')^{-1}n_1p_{x_1}n_1^*n_1'p_{x_1}(n_1')^*
\\
&\qquad=\big(h_E(x_1)(n_1^*n_1)^{-1}n_1p_{x_1}n_1^*\big)\big(h_E(x_1)((n_1')^*n_1')^{-1}
n_1'p_{x_1}(n_1')^*\big)\\
&\qquad=p_{\alpha_{n_1}(x_1)},
\end{align*}
where the last equality follows from Lemma~\ref{lem: properties of alpha 
n}(a). Hence $[U_{(\alpha_{n_1}(x_1),n_1^*,(n_1')^*)}]_1=0$.
\end{proof}

We equip $\grp_{(C^*(E),\mathcal{D}(E))}$ with the topology generated by 
$\{\{[(n,x)]:x\in\dom(n)\}:n\in N(\mathcal{D}(E))\}$. It can be proven directly that $\grp_{(C^*(E),\mathcal{D}(E))}$ is a topological groupoid with this topology, however, it also follows from our next result. 

\begin{prop}\label{prop:weyl}
Let $E$ be a graph. Then $\grp_{(C^*(E),\mathcal{D}(E))}$ is a topological groupoid, and $\grp_{(C^*(E),\mathcal{D}(E))}$ and $\grp_E$ are 
isomorphic as topological groupoids.
\end{prop}

\begin{remark}\label{rem: the Weyl groupoid case}
If $\grp_E$ is topological principally, which we know from Proposition 
~\ref{prop:topological principal} is equivalent to $E$ satisfying condition 
(L), then $\grp_{(C^*(E),\mathcal{D}(E))}$ is isomorphic to the Weyl groupoid 
$\grp_{C_0(\grp_E^0)}$ of $(C^*(\grp_E),C_0(\grp_E^0))$ as in \cite{Ren2}. In 
this case the 
isomorphism of $\grp_E$ and $\grp_{(C^*(E),\mathcal{D}(E))}$ proved below 
follows from \cite[Proposition 4.14]{Ren2}.
\end{remark}

To prove Proposition~\ref{prop:weyl} we need the following result. The proof can be deduced from the proof of \cite[Proposition 
4.8]{Ren2}, but we include a proof for completeness. As in \cite{Ren2}, we let $\supp'(f):=\{y\in\grp_E:f(\gamma)\ne 0\}$ for $f\in C^*(\grp_E)$.

\begin{lemma}\label{lem: support of f}
Let $E$ be a graph and $\pi:C^*(E)\to C^*(\grp_E)$ the isomorphism from Proposition~\ref{prop:groupoid}. Let $n\in N(\DD(E)$, and $f:=\pi(n)$. Then $\supp'(f)$ 
satisfies  
\begin{enumerate}
\item[(i)] $s(\supp'(f))=\dom(n)$;
\item[(ii)] $(x,k,y)\in\supp'(f)\Longrightarrow \alpha_n(y)=x$; and
\item[(iii)] $y\in\dom(n)\Longrightarrow (\alpha_n(y),k,y)\in \supp'(f)$ for some $k\in\Z$. 
\end{enumerate} 
\end{lemma}

\begin{proof}
Identity (i) follows because 
\[
h_E(y)(n^*n)=\pi(n^*n)(y,0,y)=f^*f(y,0,y)=\sum_{\substack{\gamma\in\GG_E \\ 
s(\gamma)=(y)}}|f(\gamma)|^2.
\]

For (ii) we first consider the function $f^*gf$ where $g$ is any element of 
$\pi(\DD(E))=C_0(\GG_E^{(0)})$. Using the convolution product we have
\begin{equation}\label{eq: key cont convolution}
f^*gf(y,0,y)=\sum_{\substack{\gamma\in\GG_E \\ 
s(\gamma)=(y)}}|f(\gamma)|^2g(r(\gamma))\quad 
\text{for all $x\in\partial E$}.
\end{equation}
Alternatively, we can also apply (\ref{eq: defining prop of alphas}) to, say, 
$g=\pi(d)$ to get
\begin{align}
f^*gf(y,0,y)=\pi(n^*dn)(y,0,y)=h_E(y)(n^*dn)&=h_E(\alpha_n(y))(d)h_E(y)(n^*n)\nonumber\\
&= g(\alpha_n(y),0,\alpha_n(y))|f(y,0,y)|^2.\label{eq: key cont hE}
\end{align}
Now suppose for contradiction that $(x,k,y)\in\supp'(f)$ but 
$\alpha_n(y)\not=x$. Choose $g\in C_0(\GG_E^{(0)})$ a positive function with 
$g(x,0,x)=1$ and 
$g(\alpha_n(y),0,\alpha_n(y))=0$. Then (\ref{eq: key cont convolution}) gives 
\[
f^*gf(y,0,y)\ge |f(x,k,y)|^2g(x,0,x)>0, 
\]
whereas (\ref{eq: key cont hE}) gives
\[
f^*gf(y,0,y)=g(\alpha_n(y),0,\alpha_n(y))|f(y,0,y)|^2=0.
\]
So (ii) holds. 

Implication (iii) follows immediately from (i) and (ii).
\end{proof}

\begin{proof}[Proof of Proposition~\ref{prop:weyl}]
Let $(x,k,y)\in\grp_E$. Then there are $\mu,\nu\in E^*$ and $z\in\partial E$ 
such that $x=\mu z$, $y=\nu z$, and $k=|\mu|-|\nu|$. We know from 
Lemma~\ref{lem: alpha sub spanning elt } that $s_\mu s_\nu^*\in 
N(\mathcal{D}(E))$, $y\in\dom(s_\mu s_\nu^*)$, and that 
$\alpha_{s_\mu s_\nu^*}(y)=x$. Define 
$\phi:\grp_E\to\grp_{(C^*(E),\mathcal{D}(E))}$ by 
\[
\phi((x,k,y))=[(s_\mu s_\nu^*,y)].
\]
It is routine to check that $\phi$ is well-defined, in the sense that if 
$\mu,\nu,\mu',\nu'\in E^*$, $z,z'\in\partial E$, $\mu z=\mu' z'$, $\nu z=\nu' 
z'$, and $|\mu|-|\nu|=|\mu'|-|\nu'|$, then $[(s_\mu s_\nu^*, \nu 
z)]=[(s_{\mu'} s_{\nu'}^*,\nu' z')]$. It is also routine to check that 
$\phi$ is a groupoid homomorphism. We now have to show that $\phi$ is a 
homeomorphism.
	
To show that $\phi$ is injective, assume that 
$\phi((x,k,y))=\phi((x',k',y'))$. Then $x=x'$ and $y=y'$. Suppose for 
contradiction that $k\ne k'$. Then $y$ must be eventually periodic, because 
otherwise we would have $\alpha_{s_\kappa s_\lambda^*}(y)\ne 
\alpha_{s_{\kappa'}s_{\lambda'}^*}(y)$ for $|\kappa|-|\lambda|=k$ and $|\kappa'|-|\lambda'|=k'$. 
Thus $x=\mu \eta\eta\eta\cdots$ and 
$y=\nu\eta\eta\eta\cdots$ for some $\mu,\nu\in E^*$ and a simple loop $\eta\in E^*$ 
such that $s(\eta)=r(\mu)=r(\nu)$. It follows that 
$\phi((x,k,y))=[(s_{\mu(\eta)^m}s_{\nu (\eta)^n}^*,y)]$ and 
$\phi((x,k',y))=[(s_{\mu(\eta)^{m'}}s_{\nu (\eta)^{n'}}^*,y)]$ where 
$m,n,m',n'$ are nonnegative integers such that $|\mu 
(\eta)^m|-|\nu(\eta)^n|=k$ and $|\mu (\eta)^{m'}|-|\nu(\eta)^{n'}|=k'$. 
Suppose that $\eta$ has an exit. Then $y\notin\partial E_\is$, and there is a 
$\zeta\in E^*$ such that $s(\zeta)=s(\eta)$, $|\zeta|\le |\eta|$, and 
$\zeta\ne \eta_1\eta_2\cdots \eta_{|\zeta|}$ (where $\eta=\eta_1\eta_2\cdots 
\eta_{|\eta|}$). Then for any open set $U$ with $y\in 
U\subseteq\dom(s_{\mu(\eta)^m}s_{\nu 
(\eta)^n}^*)\cap\dom(s_{\mu(\eta)^{m'}}s_{\nu (\eta)^{n'}}^*)$, there is a positive 
integer $l$ such that $\emptyset\ne Z(\nu(\eta)^l\zeta)\subseteq U$, and that 
$\alpha_{s_{\mu(\eta)^m}s_{\nu (\eta)^n}^*}(z)\ne 
\alpha_{s_{\mu(\eta)^{m'}}s_{\nu (\eta)^{n'}}^*}(z)$ for any $z\in 
Z(\mu(\eta)^l\zeta)$. This contradicts the assumption that 
$\phi((x,k,y))=\phi((x,k',y))$. If $\eta$ does not have an exit, then 
$y\in\partial E_\is$. Without loss of generality assume $k > k'$, then we can use Lemma \ref{lem: properties of alpha n}(c) to compute
\[
[p_y(s_{\mu(\eta)^m}s_{\nu (\eta)^n}^*)^*s_{\mu(\eta)^{m'}}s_{\nu 
(\eta)^{n'}}^*p_y]_1=[p_ys_{\nu}s_{\eta^{k-k'}}s_{\nu}^*p_y]_1=[(p_ys_{\nu}s_{\eta}s_{\nu}^*p_y)^{k-k'}]_1,
\]
and the second assertion in Lemma \ref{lem:corner} implies that 
$[(p_ys_{\nu}s_{\eta}s_{\nu}^*p_y)^{k-k'}]_1\neq 0$. 
Thus, 
\[
[U_{(y , s_{\mu(\eta)^m}s_{\nu (\eta)^n}^* , s_{\mu(\eta)^{m'}}s_{\nu 
(\eta)^{n'}}^*)}]_1 \neq 0
\]
and hence $(s_{\mu(\eta)^m}s_{\nu (\eta)^n}^*,y) \not\sim 
(s_{\mu(\eta)^{m'}}s_{\nu (\eta)^{n'}}^*,y)$. But this means 
$\phi((x,k,y))\not=\phi((x,k',y))$, which is a contradiction. So we must have 
$k=k'$, and hence $\phi$ is injective.
	
To show that $\phi$ is surjective, let $[(n,x)]$ be an arbitrary 
element of $\grp_{(C^*(E),\mathcal{D}(E))}$. Let $f:=\pi(n)$, where 
$\pi:C^*(E)\to C^*(\grp_E)$ is the isomorphism from 
Proposition~\ref{prop:groupoid}. We know from (iii) of Lemma~\ref{lem: 
support of f} that $(\alpha_n(x),k,x)\in\supp'(f)$ for some $k\in\Z$. Suppose 
first that 
$x\notin\partial E_\is$.  Choose $\mu,\nu\in E^*$, a clopen neighborhood $U$ of 
$\alpha_n(x)$, and a clopen neighborhood $V$ of $x$ such that $U\subseteq 
Z(\mu)$, $V\subseteq Z(\nu)$, $\sigma_E^{|\mu|}(U)=\sigma_E^{|\nu|}(V)$, $k=|\mu|-|\nu|$, and 
$Z(U,|\mu|,|\nu|,V)\subseteq\supp'(f)$. Then $\alpha_{s_\mu s_\nu^*}(y)=\alpha_n(y)$ 
for all $y\in V$, and hence $\phi(\alpha_n(x),|\mu|-|\nu|,x)=[(s_\mu 
s_\nu^*,x)]=[(n,x)]$. 

Now suppose that $x\in\partial E_\is$ is not eventually periodic. Choose $\mu,\nu\in E^*$ and $z\in\partial E$ such that $x=\nu z$, 
$\alpha_n(x)=\mu z$, and $k=|\mu|-|\nu|$. It follows from Lemma 
\ref{lem:corner} that $[U_{(x,n,s_{\mu}s_\nu^*)}]_1=0$ (because $K_1(\C)=0$), and thus that 
$\phi((\alpha_n(x),k,x))=[(s_\mu s_\nu^*,x)]=[(n,x)]$. Assume then that $x$ is 
eventually periodic. Then there are $\mu,\nu\in E^*$ and a simple loop 
$\eta\in E^*$ such that $s(\eta)=r(\mu)=r(\nu)$, $x=\nu\eta\eta\eta\cdots$, 
$\alpha_n(x)=\mu\eta\eta\eta\cdots$, and $k=|\mu|-|\nu|$. Choose positive 
integers $l$ and $m$ such that $[U_{(x,n,s_{\mu}s_\nu^*)}]_1=l-m$. Then by 
Lemma~\ref{lem:corner} we have
\[
[p_xs_\mu 
s_{\eta^l}s_{\eta^m}^*s_\nu^*p_x]_1=[(p_xs_\mu s_\eta s_\nu^*p_x)^l(p_xs_\mu 
s_\eta^*s_\nu^*p_x)^m]_1=l-m,
\]
and hence $[U_{(x,n,s_{\mu}s_\nu^*)}]_1=[p_xs_\nu 
s_{\eta^l}s_{\eta^m}^*s_\nu^*p_x]_1$. Since 
\[
h_E(s_{\nu\eta^m}s_{\nu\eta^m}^*)^{-1/2}U_{(x,n,s_\mu s_\nu^*)}(p_xs_\nu 
s_{\eta^l}s_{\eta^m}^*s_\nu^*p_x)^*=U_{(x,n,s_{\mu\eta^m}s_{\nu\eta^l}^*)},
\]
We have $[U_{(x,n,s_{\mu\eta^m}s_{\nu\eta^l}^*)}]_1=[U_{(x,n,s_\mu 
s_\nu^*)}]_1-[p_xs_\nu 
s_{\eta^l}s_{\eta^m}^*s_\nu^*p_x]_1=0.$ Hence 
\[
\phi((\alpha_n(x),|\mu 
(\eta)^m|-|\nu 
(\eta)^l|,x))=[(s_{\mu(\eta)^m}s_{\nu(\eta)^l}^*,x)]=[(n,x)],
\]
which shows 
that 
$\phi$ is surjective.

To see that $\phi$ is open, let $\mu,\nu\in E^*$ and let $U$ and $V$ 
be clopen subsets of $\partial E$ such that $U\subseteq Z(\mu)$, $V\subseteq 
Z(\nu)$, and $\sigma_E^{|\mu|}(U)=\sigma_E^{|\nu|}(V)$. Then there is a 
$p_V\in\mathcal{D}(E)$ such that $h_E(x)(p_V)=1$ if $x\in V$, and 
$h_E(x)(p_V)=0$ if $x\in\partial E\setminus V$; and then 
$\phi(Z(U,|\mu|,|\nu|,V))=\{[s_\mu s_\nu^* p_V,x]:x\in\dom(s_\mu s_\nu^* 
p_V)\}$. This shows that $\phi$ is open.
	
To prove that $\phi$ is continuous we will show that 
$\phi^{-1}(\{[(n,y)]:y\in \dom(n)\})$ is open for each $n\in N(\DD(E))$. Fix 
$n\in 
N(\mathcal{D}(E))$ 
and 
$z\in\dom(n)$. We claim that there is an open subset $V_{(n,z)}$ in $\GG_E$ 
such that 
\[
\phi^{-1}([(n,z)])\in V_{(n,z)}\subseteq \phi^{-1}(\{[(n,y)]:y\in 
\dom(n)\}).
\]
Let $\pi:C^*(E)\to C^*(\GG_E)$ be the isomorphism from 
Proposition~\ref{prop:groupoid}, and $f:=\pi(n)$. We know from (iii) of 
Lemma~\ref{lem: support of f} that $(\alpha_n(z),k,z)\in\supp'(f)$ for some 
$k\in\Z$. Suppose that there are two 
different 
integers $k_1$ and $k_2$ such that both $(\alpha_n(z),k_1,z)$ and 
$(\alpha_n(z),k_2,z)$ belong to $\supp'(f)$. Then there are 
$\mu_1,\nu_1,\mu_2,\nu_2\in E^*$ such that $(\alpha_n(z),k_1,z)\in 
Z(\mu_1,\nu_1)$, $(\alpha_n(z),k_2,z)\in Z(\mu_2,\nu_2)$ and 
$Z(\mu_1,\nu_1),Z(\mu_2,\nu_2)\subseteq \supp'(f)$. Without loss of generality 
we can 
assume that $|\mu_1|=|\mu_2|$, and then we have $\mu_1=\mu_2$. We also have 
$\nu_1=\nu_2\xi$ or $\nu_2=\nu_1\xi$ for some $\xi\in E^*\setminus E^0$; we 
assume that $\mu_2=\mu_1\xi$ and denote $\mu:=\mu_1=\mu_2$. We claim that $z$ 
is an isolated point, and that 
\[
s(Z(\mu,\nu_1))\cap 
s(Z(\mu,\nu_1\xi))=\{(z)\}.
\]   
To see this, suppose $(x)\in s(Z(\mu,\nu_1))\cap 
s(Z(\mu,\nu_1\xi))$. Then $x=\nu_1y$ for some $y$ such that $\alpha_n(y)=\mu 
y$, and $x=\nu_1\xi y'$ for some $y'$ such that $\alpha_n(y)=\mu y'$. It 
follows that $y'=y=\xi y'$, and hence $y=\nu_1\xi\xi\xi\dots$. So  
\[
s(Z(\mu,\nu_1))\cap 
s(Z(\mu,\nu_1\xi))=\{(\nu_1\xi\dots)\}=\{(z)\},
\]
and hence $z$ is an isolated point. Now 
$\phi^{-1}([(n,z)])=(\mu\xi\dots,|\mu|-|\nu_1|,\nu_1\xi\dots)$ is isolated 
because 
$\{\phi^{-1}([(n,z)])\}=Z(\{\mu\xi\dots\},|\mu|,|\nu_1|,\{\nu_1\xi\dots\})$ is 
open. So in this case we take $V_{(n,z)}=\{\phi^{-1}([(n,z)])\}$.

Now assume that there is a unique $k$ such that 
$(\alpha_n(z),k,z)\in\supp'(f)$. Choose $\mu,\nu\in E^*$ with $r(\mu)=r(\nu)$ 
and an open subset $V\subseteq r(\mu)\partial E$ such that 
$(\alpha_n(z),k,z)\in Z(\mu 
V,|\mu|,|\nu|,\nu V)\subseteq\supp'(f)$. Lemma~\ref{lem: support of f} implies 
that 
$\alpha_n(x)=\alpha_{s_\mu s_\nu^*}(x)$ for all $x\in \nu V$. We aim to find 
an open subset $W\subseteq \nu V$ such that $z\in W$ 
and 
\begin{equation}\label{eq: requirements of W}
[U_{(x,n,s_\mu s_\nu^*)}]_1=0\text{ for all }x\in W\cap\partial 
E_\is;
\end{equation} 
for then we have $\phi((\alpha_n(x),|\mu|-|\nu|,x))=[(s_\mu 
s_\nu^*,x)]=[(n,x)]$ for all $x\in W$, and 
the open subset 
$V_{(n,z)}:=Z(\alpha_n(W),|\mu|,|\nu|,W)$ 
satisfies the 
desired $\phi^{-1}([(n,z)])\in V_{(n,z)}\subseteq \phi^{-1}(\{[(n,y)]:y\in 
\dom(n)\})$.

Let $\delta:=|f(\alpha_n(z),k,z)|$. Then 
\[
h_E(z)(n^*n)=f^*f(z,0,z)=\sum_{\substack{\gamma\in\GG_E \\ s(\gamma)=(z)}}
|f(\gamma)|^2=\delta^2.
\]
Choose an open subset $V_0\subseteq \nu V$ such that $z\in V_0$ and 
$h_E(x)(n^*n)>(\delta/2)^2$ for all $x\in V_0$. Define
\[
g:=f 
^*1_{Z(\mu,\nu)}-\lambda(f^*f)^{1/2},
\]
where $\lambda=\overline{f(\alpha_n(z),k,z)}/|f(\alpha_n(z),k,z)|\in\T$. We 
claim that $g(z,j,z)=0$ for all $j\in\Z$. When $j=0$ we 
have 
\[
g(z,0,z)=\sum_{\gamma_1\gamma_2=(z,0,z)}f^*(\gamma_1)1_{Z(\mu,\nu)}(\gamma_2) 
- 
\lambda\sum_{\eta_1\eta_2=(z,0,z)}(f^*(\eta_1)f(\eta_2))^{1/2}
.\]
Implication (ii) of Lemma \ref{lem: support of f} ensures that the only terms in the sums which produce nonzero 
entries are $\gamma_1,\eta_1=(z,-k,\alpha_n(z))$ and 
$\gamma_2,\eta_2=(\alpha_n(z),k,z)$. Hence
\[
g(z,0,z)=\overline{f(\alpha_n(z),k,z)}-\lambda|f(\alpha_n(z),k,z)|=0.
\]
When $j\not=0$, both terms in the expression for $g$ contain 
$f(\alpha_n(z),k-j,z)$, which is zero. Hence $g(z,j,z)=0$. 

Use Proposition~\ref{prop:cesaro} to choose $m\in\N$ such that 
$\|g-\Sigma_m(g)\|<\delta/2$. Since $g(z,j,z)=0$ for all $j\in\Z$, there is an 
open set $W$ such that $z\in W\subseteq V_0$ and 
\begin{equation}\label{eq: g funct identity}
\Big|\Big(1-\frac{|j|}{m+1}\Big)g(x,j,x)\Big|<\frac{\delta}{2(m+1)}
\end{equation}
for all $-m\le j\le m$ and $x\in W$. Then $|\Sigma_m(g)(x,j,x)|<\delta/2$ for 
all $(x,j,x)\in\GG_E$ with $x\in W$. It follows from the definition of the 
norm on $C^*(\GG_E)$ that for all $x\in W\cap\partial E_\is$ we have
\begin{align*}
\|\pi(p_x)\Sigma_m(g)\pi(p_x)\| &\le \Big|\sum_{\gamma\in\GG_E}(\pi(p_x)
\Sigma_m(g)\pi(p_x))(\gamma)\Big|\\
&= \Big|\sum_{\gamma\in\GG_E}\sum_{\gamma_1\gamma_2\gamma_3=\gamma}
1_{\{(x,0,x)\}}(\gamma_1)\Sigma_m(g)(\gamma_2)1_{\{(x,0,x)\}}(\gamma_3)\Big|\\
&=\Big|\sum_{j\in\Z}\Sigma_m(g)(x,j,x)\Big| <\frac{\delta}{2}.
\end{align*}
Hence
\begin{equation}\label{eq: norm estimate for cty}
\|\pi(p_x)g\pi(p_x)\|\le\|g-\Sigma_m(g)\|+\|\pi(p_x)\Sigma_m(g)\pi(p_x)\|<\delta.
\end{equation}
We now claim that $\|U_{(x,n,s_\mu s_\nu^*)}-\lambda p_x\|<2$ for all $x\in 
W\cap\partial E_\is$. To see this, first note that 
\[
\pi(p_x)=(f^*f)^{-1/2}(x,0,x) (f^*f)^{1/2}(x,0,x) \pi(p_x)=(f^*f)(x,0,x)^{-1/2}(f^*f)^{1/2} 
\pi(p_x).
\]
Thus
\begin{align*}
\pi(U_{(x,n,s_\mu s_\nu^*)}-\lambda p_x) &= 
\pi\big(h_E(x)(n^*n)^{-1/2}p_xn^*s_\mu s_\nu^* p_x-\lambda p_x\big)\\
&= (f^*f)(x,0,x)^{-1/2}\pi(p_x)f^*1_{Z(\mu,\nu)}\pi(p_x)-\lambda\pi(p_x)\\
&= 
(f^*f)(x,0,x)^{-1/2}\big(\pi(p_x)f^*1_{Z(\mu,\nu)}\pi(p_x)-\lambda(f^*f)^{1/2} 
\pi(p_x)\big)\\
&=(f^*f)(x,0,x)^{-1/2}\pi(p_x)g\pi(p_x).
\end{align*}
Using (\ref{eq: norm estimate for cty}) we now get
\begin{align*}
\|U_{(x,n,s_\mu s_\nu^*)}-\lambda p_x\|=\|\pi(U_{(x,n,s_\mu s_\nu^*)}-\lambda 
p_x)\|&=(f^*f)(x,0,x)^{-1/2}\|\pi(p_x)g\pi(p_x)\|\\
&<(f^*f)(x,0,x)^{-1/2}\delta.
\end{align*}
Recall that $x\in W\cap \partial E_\is\subseteq V_0$, and hence 
$(f^*f)(x,0,x)^{-1/2}=h_E(x)(n^*n)^{-1/2}<2/\delta$. So 
\[
\|U_{(x,n,s_\mu s_\nu^*)}-\lambda p_x\|<2.
\]
But this means $[U_{(x,n,s_\mu s_\nu^*)}]_1=0$, and so $W$ satisfies the 
desired (\ref{eq: requirements of W}). As mentioned, this means 
$V_{(n,z)}:=Z(\alpha_n(W),|\mu|,|\nu|,W)$ 
satisfies 
\[
\phi^{-1}([(n,z)])\in V_{(n,z)}\subseteq \phi^{-1}(\{[(n,y)]:y\in 
\dom(n)\}),
\]
as required. 
\end{proof}

\begin{prop}\label{prop:groupoid-iso}
	Let $E$ and $F$ be two graphs. If there is an isomorphism from $C^*(E)$ to 
	$C^*(F)$ which maps $\mathcal{D}(E)$ to $\mathcal{D}(F)$, then 
	$\grp_{(C^*(E),\mathcal{D}(E))}$ and $\grp_{(C^*(F),\mathcal{D}(F))}$ are isomorphic as topological groupoids, and 
	consequently $\grp_E$ and $\grp_F$ are isomorphic as topological groupoids.
\end{prop}

\begin{proof}
	Suppose $\phi$ is an isomorphism from $C^*(E)$ to $C^*(F)$ which maps 
	$\mathcal{D}(E)$ to $\mathcal{D}(F)$. Then there is a homeomorphism 
	$\kappa:\partial E\to\partial F$ such that 
	$h_E(x)(f)=h_F(\kappa(x))\phi(f)$ for all $f\in\mathcal{D}(E)$ and all 
	$x\in\partial E$. It is routine to check that the map 
	$[(n,x)]\mapsto [(\phi(n),\kappa(x))]$ is an isomorphism between the 
	topological groupoids $\grp_{(C^*(E),\mathcal{D}(E))}$ and 
	$\grp_{(C^*(F),\mathcal{D}(F))}$. Then Proposition 
	\ref{prop:weyl} implies that $\grp_E$ and $\grp_F$ are isomorphic as 
	topological groupoids.
\end{proof}

\section{Main result and examples}\label{sec: Statement of the main result}

\begin{thm}\label{thm: the main thm}
Let $E$ and $F$ be graphs. Consider the following four statements. 
\begin{itemize}
\item[(1)] There is an isomorphism from $C^*(E)$ to $C^*(F)$ which maps $\mathcal{D}(E)$ onto $\mathcal{D}(F)$.
\item[(2)] The graph groupoids $\grp_E$ and $\grp_F$ are isomorphic as topological groupoids.
\item[(3)] The pseudogroups of $E$ and $F$ are isomorphic.
\item[(4)] $E$ and $F$ are orbit equivalent.
\end{itemize}
Then $(1)\iff (2)$, $(3)\iff (4)$ and $(2)\implies (3)$. If $E$ and $F$ 
satisfy condition (L), then $(3)\implies (2)$ and the four statements are 
equivalent.
\end{thm}

\begin{proof}
	$(1)\implies (2)$ is proved in Proposition \ref{prop:groupoid-iso}. $(2)\implies (1)$ follows from Proposition \ref{prop:groupoid}. $(3)\iff (4)$ is proved in Proposition \ref{prop:pseudogroups and orbit equivalence}. $(2)\implies (3)$ follows directly from the definition of the pseudogroups $\mathcal{P}_E$ and $\mathcal{P}_F$.
	
	 Assume that $E$ and $F$ satisfy condition (L). Then it follows from Proposition \ref{prop:topological principal} and \cite[Proposition 3.6(i)]{Ren2} that $\grp_E$ is isomorphic to the groupoid of germs of the pseudogroup $\mathcal{P}_E$ constructed on page 8 of \cite{Ren2}, and that $\grp_F$ is isomorphic to the groupoid of germs of the pseudogroup $\mathcal{P}_F$. It follows that if $\mathcal{P}_E$ and $\mathcal{P}_F$ are isomorphic, then $\grp_E$ and $\grp_F$ are isomorphic. Thus $(3) \implies (2)$, and all 4 statements are equivalent when $E$ and $F$ satisfy condition (L).
\end{proof}

\begin{example}\label{ex: countereg to 3 implies 2}
We show that (3) does not imply (2) in general. Consider the single vertex and single loop graphs
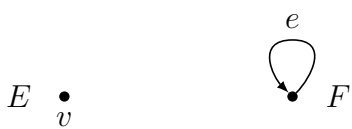
\[
\begin{tikzpicture}
    \def\vertex(#1) at (#2,#3){
        \node[inner sep=0pt, circle, fill=black] (#1) at (#2,#3)
        [draw] {.}; 
 }
    \vertex(11) at (0, 0)
    
    \vertex(21) at (3, 0);
    
    \node at (0,-0.3) {$v$};
    \node at (3,1) {$e$};
        \node at (-0.6,0) {$E$};
         \node at (3.6,0) {$F$};
    
    \draw[style=semithick, -latex] (21.north east)
        .. controls (3.25,0.25) and (3.5,0.75) ..
        (3,0.75)
        .. controls (2.5,0.75) and (2.75,0.25) ..
        (21.north west);
        
\end{tikzpicture}
\]
We have $\partial E=\{v\}$ and $\partial F=\{ee\dots\}$. So $E$ and $F$ are 
orbit 
equivalent, but $C^*(E)\cong\C$ is not isomorphic to $C^*(F)\cong C(\T)$. 
Obviously $F$ 
does not satisfy condition (L), so $E$ and $F$ provide a simple counterexample 
to the equivalence of statements (1) and (4) of Theorem~\ref{thm: the main 
thm} without the presence of condition (L).
\end{example}

\begin{example}\label{ex: another countereg to 3 implies 2}
The graphs
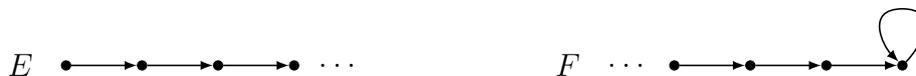
\[
\begin{tikzpicture}
    \def\vertex(#1) at (#2,#3){
        \node[inner sep=0pt, circle, fill=black] (#1) at (#2,#3)
        [draw] {.}; 
 }
    \vertex(11) at (0,0)
    
    \vertex(21) at (1,0)
    
    \vertex(31) at (2,0)
    
    \vertex(41) at (3,0)
    
    \vertex(51) at (8,0)
    
    \vertex(61) at (9,0)
    
    \vertex(71) at (10,0)
    
    \vertex(81) at (11,0);
    
\node at (-0.6,0) {$E$};
\node at (6.6,0) {$F$};

\node at (3.6,0) {$\dots$};
\node at (7.4,0) {$\dots$};

\draw[style=semithick, -latex] (11.east)--(21.west);
\draw[style=semithick, -latex] (21.east)--(31.west);
\draw[style=semithick, -latex] (31.east)--(41.west);

\draw[style=semithick, -latex] (51.east)--(61.west);
\draw[style=semithick, -latex] (61.east)--(71.west);
\draw[style=semithick, -latex] (71.east)--(81.west);

\draw[style=semithick, -latex] (81.north east)
        .. controls (11.25,0.25) and (11.5,0.75) ..
        (11,0.75)
        .. controls (10.5,0.75) and (10.75,0.25) ..
        (81.north west);
        
\end{tikzpicture}
\]
provide a similar counterexample to the equivalence of statements (1) and (4) 
of Theorem~\ref{thm: the main 
thm} without the presence of condition (L). In this case $\partial 
E=\N=\partial F$ (and, unlike Example~\ref{ex: countereg to 3 implies 2}, the 
shift map is defined on all of $\partial E$ and $\partial F$), but 
$C^*(E)\cong\KK\not\cong \KK\otimes C(\T)\cong C^*(F)$.   
\end{example}

\begin{example}\label{ex: O2 and O2-}
	There exist graphs $E$ and $F$ such that $C^*(E)$ and $C^*(F)$ are isomorphic, and $\mathcal{D}(E)$ and $\mathcal{D}(F)$ isomorphic, but $E$ and $F$ are not orbit equivalent.
	
Consider for example the graphs $E_2$ and $E_{2}^{-}$ below.
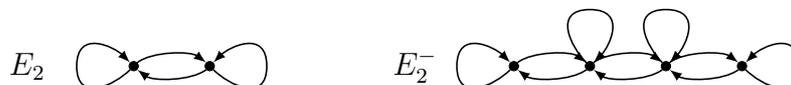
\[
\begin{tikzpicture}
    \def\vertex(#1) at (#2,#3){
        \node[inner sep=0pt, circle, fill=black] (#1) at (#2,#3)
        [draw] {.}; 
 }
    \vertex(11) at (-1,0)
    
    \vertex(21) at (0,0)
    
    \vertex(31) at (4,0)
    
    \vertex(41) at (5,0);
	
	\vertex(51) at (6,0);
	
	\vertex(61) at (7,0);
    
\node at (-2.4,0) {$E_2$};
\node at (2.7,0) {$E_{2}^{-}$};

\draw[style=semithick, -latex] (11.south west)
.. controls (-1.25,-0.25) and (-1.75,-0.5) ..
 (-1.75,0) 
 .. controls (-1.75,0.5) and (-1.25,0.25) .. (11.north west);

\draw[style=semithick, -latex] (21.south east)
.. controls (0.25,-0.25) and (0.75,-0.5) ..
 (0.75,0) 
 .. controls (0.75,0.5) and (0.25,0.25) .. (21.north east);
 
 \draw[style=semithick, -latex] (11.north east)
 .. controls (-0.75,0.2) and (-0.25,0.2).. (21.north west); 
 
 \draw[style=semithick, -latex] (21.south west)
 .. controls (-0.25,-0.2) and (-0.75,-0.2).. (11.south east);
 
 \draw[style=semithick, -latex] (31.north east)
 .. controls (4.25,0.2) and (4.75,0.2).. (41.north west); 
 
 \draw[style=semithick, -latex] (41.south west)
 .. controls (4.75,-0.2) and (4.25,-0.2).. (31.south east);

 \draw[style=semithick, -latex] (41.north east)
 .. controls (5.25,0.2) and (5.75,0.2).. (51.north west); 
 
 \draw[style=semithick, -latex] (51.south west)
 .. controls (5.75,-0.2) and (5.25,-0.2).. (41.south east);
 
 \draw[style=semithick, -latex] (51.north east)
 .. controls (6.25,0.2) and (6.75,0.2).. (61.north west); 
 
 \draw[style=semithick, -latex] (61.south west)
 .. controls (6.75,-0.2) and (6.25,-0.2).. (51.south east);
 
 \draw[style=semithick, -latex] (31.south west)
 .. controls (3.75,-0.25) and (3.25,-0.5) ..
  (3.25,0) 
  .. controls (3.25,0.5) and (3.75,0.25) .. (31.north west);
  
  \draw[style=semithick, -latex] (61.south east)
  .. controls (7.25,-0.25) and (7.75,-0.5) ..
   (7.75,0) 
   .. controls (7.75,0.5) and (7.25,0.25) .. (61.north east);
   
   \draw[style=semithick, -latex] (41.north west)
   .. controls (4.75,0.25) and (4.5,0.75) ..
    (5,0.75) 
    .. controls (5.5,0.75) and (5.25,0.25) .. (41.north east);
	
    \draw[style=semithick, -latex] (51.north west)
    .. controls (5.75,0.25) and (5.5,0.75) ..
     (6,0.75) 
     .. controls (6.5,0.75) and (6.25,0.25) .. (51.north east);
\end{tikzpicture}
\]
It follows from \cite[Remark 2.8]{Raeburn2005} that the $C^*$-algebra of $E_2$ is isomorphic to $\mathcal{O}_2$ (see for example \cite{Ror}) and that the $C^*$-algebra of $E_{2}^{-}$ is isomorphic to $\mathcal{O}_{2}^{-}$ (see for example \cite{Ror}). It is proved in \cite[Lemma 6.4]{Ror} that $\mathcal{O}_2$ and $\mathcal{O}_{2}^{-}$ are isomorphic. We also have that $\mathcal{D}(E_2)$ and $\mathcal{D}(E_{2}^{-})$ because both $\partial E_2$ and $\partial E_{2}^{-}$ are Cantor sets. However, $E_2$ and $E_{2}^{-}$ cannot be orbit equivalent because if they were, then it would follow from Theorem \ref{thm: the main thm} and \cite[Theorem 3.6]{MM} that $\det(I-A_2)=\det(I-A_{2}^{-})$ where 
\begin{equation*}
A_2=\begin{pmatrix}
	1&1\\
	1&1
\end{pmatrix} \qquad \text{ and } \qquad
A_2^-=\begin{pmatrix}
	1&1&0&0\\
	1&1&1&0\\
	0&1&1&1\\
	0&0&1&1
\end{pmatrix}.
\end{equation*} 
However, $\det(I-A_2)=-1$ and $\det(I-A_{2}^{-})=1$.
\end{example}

\section{Applications}\label{sec: applications}

In this section we provide two applications of Theorem \ref{thm: the main thm}. Our first result shows that conjugacy of general graphs implies that their $C^*$-algebras are isomorphic and the isomorphism decends to their maximal abelian subalgebras. As a corollary we obtain a strengthening of \cite[Theorem 3.2]{BP}. Our second application adds three additional equivalences to \cite[Theorem 1.1]{ERS}, which provides a complete invariant for amplified graphs.

\subsection{Conjugacy and out-splitting}\label{subsec: conj and out-sp}

Two graphs $E$ and $F$ are said to be \emph{conjugate} if there is a 
homeomorphism $h:\partial E\to\partial F$ such that $h(\partial E^{\ge 
1})=\partial F^{\ge 1}$ and $h(\sigma_E(x))=\sigma_F(h(x))$ for all 
$x\in\partial E^{\ge 1}$. It is routine to verify that if $E$ and $F$ are 
conjugate, then they are also orbit equivalent. Thus Theorem \ref{thm: the 
main thm} implies that if $E$ and $F$ both satisfy condition (L) and they are 
conjugate, then there is an isomorphism from $C^*(E)$ to $C^*(F)$ which maps 
$\mathcal{D}(E)$ onto $\mathcal{D}(F)$. In Theorem \ref{thm:conj} we will 
prove that if $E$ and $F$ are conjugate, then $\grp_E$ and $\grp_F$ are 
isomorphic, and hence there is an isomorphism from $C^*(E)$ to $C^*(F)$ which 
maps $\mathcal{D}(E)$ onto $\mathcal{D}(F)$, even if $E$ and $F$ do not 
satisfy condition (L). As a corollary, we strengthen \cite[Theorem 3.2]{BP} 
for out-splittings of graphs.

\begin{thm}\label{thm:conj}
	Let $E$ and $F$ be graphs. If $E$ and $F$ are conjugate, then $\grp_E$ and $\grp_F$ are isomorphic as topological groupoids, and hence there is an isomorphism from $C^*(E)$ to $C^*(F)$ which maps $\mathcal{D}(E)$ onto $\mathcal{D}(F)$.
\end{thm}

\begin{proof}
	Let $h:\partial E\to\partial F$ be a homeomorphism such that $h(\partial E^{\ge 1})=\partial F^{\ge 1}$ and $h(\sigma_E(x))=\sigma_F(h(x))$ for all $x\in\partial E^{\ge 1}$. Define $\phi:\grp_E\to\grp_F$ by $\phi((x,k,y))=(h(x),k,h(y))$. Then $\phi$ is a homeomorphism, and $\grp_E$ and $\grp_F$ are isomorphic as topological groupoids. Then Theorem \ref{thm: the main thm} implies that there is an isomorphism from $C^*(E)$ to $C^*(F)$ which maps $\mathcal{D}(E)$ onto $\mathcal{D}(F)$.
\end{proof}

As a corollary we are able to strengthen \cite[Theorem 3.2]{BP}. Before we state the corollary we recall the terminology of \cite{BP}.

Let $E$ be a graph and let $\mathcal{P}$ be a partition of $E^1$ constructed 
in the following way. For each $v\in E^0$ with $vE^1\ne\emptyset$, partition 
$vE^1$ into disjoint nonempty subsets 
$\mathcal{E}_v^1,\dots,\mathcal{E}_v^{m(v)}$ where $m(v)\ge 1$, and let 
$m(v)=0$ when $vE^1=\emptyset$. The partion $\mathcal{P}$ is \emph{proper} if 
for each $v\in E^0$ we have that $m(v)<\infty$ and that $\mathcal{E}_v^i$ is 
infinite for at most one $i$. The \emph{out-split} of $E$ with respect to 
$\mathcal{P}$ is the graph $E_s(\mathcal{P})$ where
\begin{align*}
	E_s(\mathcal{P})^0&:=\{v^i:v\in E^0,\ 1\le i\le m(v)\}\cup\{v:v\in E^0,\ m(v)=0\},\\
	E_s(\mathcal{P})^1&:=\{e^j:e\in E^1,\ 1\le j\le m(r(e))\}\cup\{e:e\in E^1,\ m(r(e))=0\}
\end{align*} 
and $r,s:E_s(\mathcal{P})^1\to E_s(\mathcal{P})^0$ are given by
\begin{align*}
	s(e^j)&:=s(e)^i\text{ and }r(e^j):=r(e)^j\text{ for }e\in\mathcal{E}_{s(e)}^i\text{ with }m(r(e))\ge 1,\text{ and}\\
	s(e)&:=s(e)^i\text{ and }r(e):=r(e)\text{ for }e\in\mathcal{E}_{s(e)}^i\text{ with }m(r(e))=0.
\end{align*}

\begin{cor}
	Let $\mathcal{P}$ be a proper partition of $E^1$ as above. Then $E$ and $E_s(\mathcal{P})$ are conjugate and there is an isomorphism from $C^*(E)$ to $C^*(E_s(\mathcal{P}))$ which maps $\mathcal{D}(E)$ onto $\mathcal{D}(E_s(\mathcal{P}))$.
\end{cor}

\begin{proof}
	Notice that since $\mathcal{P}$ is proper, we have that $v^i\in E_s(\mathcal{P})^0_\reg$ if $v\in E^0_\reg$, and that if $vE^1$ is infinite, then $v^i E_s(\mathcal{P})^1$ is infinite for exactly one $i$. 
	
	For $x=x_0x_1\dots\in\partial E$, let $h(x)=y=y_0y_1\dots\in\partial E_s(\mathcal{P})$ be defined by $h(x)$ having the same length as $x$ and 
	\begin{equation*}
		y_n:=
		\begin{cases}
			x_n&\text{if }m(r(x_n))=0,\\
			x_n^j&\text{if }x_{n+1}\in\mathcal{E}_{r(e)}^j,\\
			x_n^j&\text{if }r(x_n)^j E_s(\mathcal{P})^1\text{ is infinite and the length of }x\text{ is }n.
		\end{cases}
	\end{equation*}
	Then the map $x\mapsto h(x)$ is a homeomorphism from $\partial E$ to $\partial E_s(\mathcal{P})$, $h(\partial E^{\ge 1})=\partial E_s(\mathcal{P})^{\ge 1}$ and $h(\sigma_E(x))=\sigma_{E_s(\mathcal{P})}(h(x))$ for all $x\in\partial E^{\ge 1}$. Thus, $E$ and $E_s(\mathcal{P})$ are conjugate and it follows from Theorem \ref{thm:conj} that there is an isomorphism from $C^*(E)$ to $C^*(E_s(\mathcal{P}))$ which maps $\mathcal{D}(E)$ onto $\mathcal{D}(E_s(\mathcal{P}))$.
\end{proof}

\begin{remark}\label{remark: back to coe but not conjugacy}
To see that orbit equivalence is weaker than conjugacy consider the graphs $E$ 
and $F$ from Example~\ref{ex: coe but not conjugacy}. We have already seen 
that $E$ and $F$ are orbit equivalent. They are not, however, conjugate 
because $\sigma_E(e_2e_2\dots)=e_2e_2\dots$ and $\sigma_F(y)\ne y$ for all 
$y\in\partial F$, and fixed points are a conjugacy invariant.
\end{remark}

\subsection{Amplified graphs and orbit equivalence}\label{subsec: amplified graphs}

In \cite{ERS}, a graph is called \emph{amplified} if whenever there is an edge between two vertices in the graph, there are infinitely many. Theorem 1.1 in \cite{ERS} characterises when the $C^*$-algebras of amplified graphs are isomorphic. Using our main result, we improve this result by adding three additional equivalences, see Theorem \ref{thm:amplified}.  Before we precisely state the result of \cite{ERS} and our improvement, we will first recall the notation of \cite{ERS}. 

If $E$ is a graph, then the \emph{amplification} of $E$ is the graph $\overline{E}$ defined by $\overline{E}^0:=E^0$, $\overline{E}^1:=\{e(v,w)^n:e\in E^1,\ s(e)=v,\ r(e)=w,\ n\in\N\}$, $s(e(v,w)^n):=v$, and $r(e(v,w)^n):=w$. It is routine to see that a graph $E$ is amplified if and only if $E=\overline{E}$. 

If $E$ is a graph, then the \emph{transitive closure} of $E$ is the graph $\mathsf{t}E$ defined by $\mathsf{t}E^0:=E^0$, $\mathsf{t}E^1:=E^1\cup\{e(v,w):\mu\in E^*\setminus (E^0\cup E^1),\ s(\mu)=v,\ r(\mu)=w\}$, with source and range maps that extend those of $E$ and satisfy $s(e(v,w)^n):=v$, and $r(e(v,w)^n):=w$.

Theorem 1.1 of \cite{ERS} says that if $E$ and $F$ are graphs with $E^0$ and $F^0$ finite, then the following 6 statements are equivalent.
	\begin{enumerate}
		\item The graphs $\at{E}$ and $\at{F}$ are isomorphic, in the sense 
		that there are bijections $\phi^0:\at{E}^0\to \at{F}^0$ and 
		$\phi^1:\at{E}^1\to \at{F}^1$ such that $s(\phi^1(e))=\phi^0(s(e))$ 
		and $r(\phi^1(e))=\phi^0(r(e))$ for all $e\in \at{E}^1$.
		\item The $C^*$-algebras $C^*(\at{E})$ and $C^*(\at{F})$ are isomorphic.
		\item The $C^*$-algebras $C^*(\overline{E})$ and $C^*(\overline{F})$ are isomorphic.
		\item The $C^*$-algebras $C^*(\overline{E})$ and $C^*(\overline{F})$ are stably isomorphic.
		\item The tempered primitive ideal spaces $\prim(C^*(\overline{E}))$ and $\prim(C^*(\overline{F}))$ are isomorphic (see \cite[Definition 4.8]{ERS}).
		\item The ordered filtered $K$-theories $FK(C^*(\overline{E}))$ and $FK(C^*(\overline{F}))$ of $C^*(\overline{E})$ and $C^*(\overline{F})$ are isomorphic (see \cite[Definition 4.4]{ERS}).
	\end{enumerate}

The following result improves on \cite[Theorem~1.1]{ERS}.

\begin{thm}\label{thm:amplified}
	Let $E$ and $F$ be graphs with $E^0$ and $F^0$ finite. Then each of the following 3 statements is equivalent to each of the statements (1)--(6) above.
	\begin{enumerate}
		\setcounter{enumi}{6}
		\item The graphs $\overline{E}$ and $\overline{F}$ are orbit equivalent.
		\item The graph groupoids $\mathcal{G}_{\overline{E}}$ and $\mathcal{G}_{\overline{F}}$ are isomorphic as topological groupoids.
		\item There exists an isomorphism from $C^*(\overline{E})$ to $C^*(\overline{F})$ which maps $\mathcal{D}(\overline{E})$ onto $\mathcal{D}(\overline{F})$.
	\end{enumerate}
	Hence the statements (1)--(9) are all equivalent.
\end{thm}

To prove Theorem~\ref{thm:amplified} we need two results. We start with a 
modification of 
\cite[Theorem 3.8]{ERS}.

\begin{lemma}\label{lem:ext}
	Let $E$ be a graph and $\mu=\mu_1\mu_2\dots,\mu_m\in E^*$. Let $F$ be the graph with $F^0:=E^0$, $F^1:=E^1\cup\{\mu^n:n\in\N\}$, and range and source maps that extend those of $E$ and satisfy $s(\mu^n):=s(\mu)$ and $r(\mu^n):=r(\mu)$. If the set $\{e\in E^1: s(e)=s(\mu),\ r(e)=r(\mu)\}$ is infinite, then $E$ and $F$ are orbit equivalent.
\end{lemma}

\begin{proof}
	Let $A:=\{e\in E^1: s(e)=s(\mu),\ r(e)=r(\mu)\}$ and assume $A$ is infinite. Then there are injective functions $\eta_1:\N\to A$ and $\eta_2:A\to A$ such that $\eta_1(\N)\cap\eta_2(A)=\emptyset$ and $\eta_1(\N)\cup\eta_2(A)=A$. For each $x\in\partial F$, let $h(x)$ be the element of $\partial E$ obtained by, for each $n\in\N$, replacing every occurence of $\mu^n$ by the path $\eta_1(n)\mu_2\mu_3\dots\mu_m$ and, for each $e\in A$, replacing every occurence of the path $a\mu_2\mu_3\dots\mu_m$ by the path $\eta_2(a)\mu_2\mu_3\dots\mu_m$. Then $x\mapsto h(x)$ is a homeomorphism from $\partial F$ to $\partial E$.
	
	Define $k_1,l_1:\partial F^{\ge 1}\to\N$ by 
	\begin{equation*}
		k_1(x):=0\text{ for all }x\in\partial F^{\ge 1},\qquad l_1(x):=
		\begin{cases}
			m&\text{if }x\in\cup_{n\in\N}Z(\mu^n),\\
			1&\text{if }x\notin\cup_{n\in\N}Z(\mu^n).
		\end{cases}
	\end{equation*}
	Then $k_1$ and $l_1$ are both continuous, and 
	$\sigma_E^{k_1(x)}(h(\sigma_F(x)))=\sigma_E^{l_1(x)}(h(x))$ for all 
	$x\in\partial F^{\ge 1}$. Similarly, define $k_1',l_1':\partial E^{\ge 
	1}\to\N$ by 
	\begin{align*}
		k_1'(y)&:=
		\begin{cases}
			m-1&\text{if }y\in\cup_{e\in\eta_1(\N)}Z(e\mu_2\mu_3\dots\mu_m),\\
			0&\text{if }y\in\cup_{e\in\eta_2(A)}Z(e\mu_2\mu_3\dots\mu_m),\\
			0&\text{if }y\notin\cup_{e\in A}Z(e\mu_2\mu_3\dots\mu_m).
		\end{cases} \\
		 l_1'(y)&:=1\text{ for all }x\in\partial F^{\ge 1}.
	\end{align*}
	Then $k_1'$ and $l_1'$ are both continuous, and $\sigma_F^{k_1'(y)}(h^{-1}(\sigma_E(y)))=\sigma_F^{l_1'(y)}(h^{-1}(y))$ for all $y\in\partial E^{\ge 1}$. This shows that $E$ and $F$ are orbit equivalent.
\end{proof}

\begin{prop}\label{prop:trans}
	Let $E$ be a graph with $E^0$ finite. Then $\overline{E}$ and $\at{E}$ are 
	orbit equivalent.
\end{prop}

\begin{proof}
	Notice that $\at{E}$ can be obtained from $\overline{E}$ by adding 
	infinitely many edges from $v$ to $w$ whenever there is a path from $v$ to 
	$w$. Thus, that $\overline{E}$ and $\at{E}$ are orbit equivalent follows 
	from finitely many applications of Lemma \ref{lem:ext}.
\end{proof}

\begin{proof}[Proof of Theorem~\ref{thm:amplified}]
	Since both $\overline{E}$ and $\overline{F}$ satisfy condition (L), it 
	follows from our main theorem that (7)--(9) are equivalent, and it is 
	obvious that (9) implies (3). Proposition \ref{prop:trans} shows that (1) 
	implies (7).
\end{proof}

\end{document}